\documentclass[reqno,12pt]{amsart}

\allowdisplaybreaks
\usepackage{verbatim}
\usepackage{amssymb}
\usepackage{amsbsy}
\usepackage{amscd}
\usepackage{amsmath}
\usepackage{amsthm}
\usepackage[mathscr]{eucal}
\usepackage{hyperref}

\usepackage{mathpazo}
\usepackage{amsfonts, mathrsfs, amscd}
\usepackage[all]{xy}
\usepackage{tikz}
\usetikzlibrary{matrix}
\usetikzlibrary{arrows}

\newtheorem{theorem}{Theorem}[section]
\newtheorem{lemma}[theorem]{Lemma}

\newtheorem{proposition}[theorem]{Proposition}
\newtheorem{corollary}[theorem]{Corollary}

\newtheorem{fjraxiom}{FJR}
\newtheorem{cohftaxiom}{CFT}
\newtheorem{unitaxiom}{U}
\newtheorem{conjecture}[theorem]{Conjecture}
\theoremstyle{remark}
\newtheorem{remark}[theorem]{Remark}

\newtheorem{definition}[theorem]{Definition}

\numberwithin{equation}{subsection}

\newcommand{\cM}{\mathcal{M}}
\newcommand{\cA}{\mathcal{A}}
\newcommand{\cW}{\mathcal{W}}

\newcommand{\cC}{\mathcal{C}}
\newcommand{\cX}{\mathcal{X}}
\newcommand{\cH}{\mathcal{H}}
\newcommand{\cL}{\mathcal{L}}
\newcommand{\cF}{\mathcal{F}}

\newcommand{\cZ}{\mathcal{Z}}

\newcommand{\CC}{\mathbb{C}}
\newcommand{\ZZ}{\mathbb{Z}}
\newcommand{\PP}{\mathbb{P}}
\newcommand{\QQ}{\mathbb{Q}}
\newcommand{\HH}{\mathbb{H}}

\newcommand{\LL}{\mathbb{L}}
\newcommand{\NN}{\mathbb{N}}

\newcommand{\sC}{\mathscr{C}}

\newcommand{\sH}{\mathscr{H}}

\newcommand{\bc}{\mathbf{c}}

\newcommand{\bt}{\mathbf{t}}
\newcommand{\bu}{\mathbf{u}}

\newcommand{\jw}{\mathfrak{j}}

\newcommand{\im}{\mathsf{i}}

\DeclareMathOperator{\ch}{ch}
\DeclareMathOperator{\id}{id}
\DeclareMathOperator{\Aut}{Aut}
\DeclareMathOperator{\SL}{SL}

\DeclareMathOperator{\cFix}{Fix}
\DeclareMathOperator{\mult}{mult}
\newcommand{\br}[1]{\left\langle#1\right\rangle}  
\newcommand{\set}[1]{\left\{#1\right\}}  
\newcommand{\bv}[1]{\mathbf{#1}}
\newcommand{\M}[2]{ { \overline{\mathcal M}_{#1, #2} } }
\newcommand{\res}{ {\mathrm{res}} }

\title{Givental--type reconstruction at a non--semisimple point}
\begin{document}

\begin{abstract}
In this paper we consider the orbifold curve, which is a quotient of an elliptic curve $\mathcal E$ by a cyclic group of order 4.
We develop a systematic way to obtain a Givental--type reconstruction of Gromov--Witten theory of the orbifold curve via the product of the Gromov--Witten theories of a point. This is done by employing mirror symmetry and certain results in FJRW theory. 
In particular, we present the particular Givental’s action giving the CY/LG correspondence between the Gromov--Witten theory of the orbifold curve $\mathcal{E} / \ZZ_4$ and FJRW theory of the pair defined by the polynomial $x^4+y^4+z^2$ and the maximal group of diagonal symmetries. The methods we have developed can easily be applied to other finite quotients of an elliptic curve.
Using Givental's action we also recover this FJRW theory via the product of the Gromov--Witten theories of a point. 
Combined with the CY/LG action we get a result in “pure” Gromov--Witten theory with the help of modern mirror symmetry conjectures.
\end{abstract}

\author{Alexey Basalaev}
\address{Universit\"at Mannheim, Lehrsthul f\"ur Mathematik VI, Seminargeb\"aude A 5, 6, 68131 Mannheim, Germany}
\email{basalaev@uni-mannheim.de}
\author{Nathan Priddis}
\address{Leibniz Universit\"at Hannover, Welfengarten 1, 30167 Hannover, Germany}
\email{priddis@math.uni-hannover.de}

\maketitle

\setcounter{tocdepth}{1}
\tableofcontents

\section{Introduction}\label{sec:intro}
Let $\M{g}{n}$ stand for the Deligne--Mumford moduli space of stable curves and $V$ be a finite--dimensional complex vector space with a pairing $\eta$. A cohomological field theory (CohFT for brevity) $\Lambda_{g,n}$ on $(V,\eta)$ is a system of linear maps
\[
\Lambda_{g,n}: V^{\otimes n} \rightarrow H^*(\M{g}{n}),
\]
subject to the certain system of axioms, for all $g,n$ where $\M{g}{n}$ exists and is non--empty.

The study of CohFTs was initiated by physicists, who distinguished some particular classes of CohFT that play an important role in mirror symmetry. These are the Saito--Givental CohFT of an isolated singularity $\tilde W$ (giving the B model) and Gromov--Witten CohFT of a Calabi--Yau variety $X$ (giving the A model). Another type of A model CohFT, conjectured in physics and later constructed by mathematicians, is now known under the name of FJRW CohFT, associated to a pair $(W, G_{max})$, where $W$ is a polynomial defining an isolated singularity and $G_{max}$ is a symmetry group of $W$.

In this paper we address two different questions about the CohFTs.

\subsection*{Classification}
Of particular interest to mathematicians was the classification of all CohFTs, understood axiomatically in a general context (see \cite{KM}). From this point of view, the CohFTs mentioned above are some very special points in the space of all CohFTs. In this general context it's convenient to work with a CohFT $\Lambda_{g,n}$ in terms of a \textit{partition function} $\cZ^\Lambda := \exp(\sum_{g \ge 0} \hbar^{g-1} \cF_g)$, where $\cF_g$ is a genus $g$ potential of the CohFT --- the generating function of the integrals of $\Lambda_{g,n}$ over $\M{g}{n}$.

An important tool when working with the CohFTs is the group action of Givental, acting on the space of all partition functions of CohFTs (cf. \cite{G}). 
Givental's action is applied to the classificational problem as follows. For an arbitrary CohFT on $(V,\eta)$, with the partition function $\cZ$ one tries to find the Givental's group element $R$, such that
\[
\cZ = \hat S \cdot \hat R \cdot \cZ^{basic},
\]
where $\hat R$ denotes the action of $R$, $\hat S$ stands for the change of the variables, $\cZ^{basic}$ is the partition function of some ``basic'' CohFT. If we have such a formula, we say that $\cZ$ is \textit{reconstructed} from $\cZ^{basic}$ via the actions of $R$ and $S$.

The canonical choice of the ``basic'' CohFT is given by the product of $\dim(V)$ Gromov--Witten theories of a point. In this case the Givental's group element above is called an \textit{$R$-matrix of the CohFT}. It was conjectured by Givental and later proved by Teleman \cite{T}, that if CohFT $\Lambda_{g,n}$ is \textit{semisimple}, there is always an $R$--matrix, reconstructing $\cZ$ from $\cZ^{basic}$.

\subsection*{Calabi--Yau/Landau--Ginzburg correspondence} 
Closely related to mirror symmetry is the phenomenon called the Calabi--Yau/Landau--Ginzburg correspondence (CY/LG for brevity). Here Givental's action also has an application.
The CY/LG correspondence is a conjecture which in this context states that for two different A model partition functions $\cZ^{GW}$ and $\cZ^{FJRW}$, being mirrors to the same B model, there is a Givental's group element $R$ and a change of the variables $\hat S$, such that
\[
  \cZ^{FJRW} = \hat S \cdot \hat R \cdot \cZ^{GW}.
\]
\\
\noindent{\bf In this paper}
we address both problems outlined above. First we show the CY/LG correpondence for one particular pair $(x^4+y^4+z^2, G_{max})$ and a particular CY orbifold $\PP^1_{4,4,2}$ by giving the elements $R$ and $S$ explicitly. Second, we give a Givental--type formula, expressing the partition function of FJRW theory of $(x^4+y^4+z^2, G_{max})$ via the partition function of the so-called ``untwisted theory''. This step is connected to the work of \cite{ChR,PrSh,ChZ}.
The partition function of the ``untwisted theory'' differs from the partition function of the product of the Gromov--Witten theories of a point just by a linear change of the variables. Due to this fact we consider it as a ``basic'' CohFT in the sense as above.

Combining these two results we get the formula, reconstructing the genus zero potential of the Gromov--Witten theory of the orbifold $\PP^1_{4,4,2}$ from the basic $\cZ^{un}$:
\[
\cF_0^{\PP^1_{4,4,2}} = \lim_{\lambda\to 0}\res_{\hbar} \ln \left( \hat S^{-1} \cdot  \ \hat R_{GW} \cdot \cZ^{un} \right),
\]
where the limit on the RHS is the so-called \textit{non-equivariant limit}. 

In this way we get the result in Gromov--Witten theory by using mirror symmetry and modern approach to singularity theory, namely FJRW theory.

The Gromov--Witten theory of $\PP^1_{4,4,2}$ is not semisimple, which fact makes the Givental--Teleman technique not applicable. Our result shows that there could be still some reconstruction in a non--semisimple case too, but from another basic CohFT. While $\cZ^{basic}$ would be the product of $9$ functions $\cZ^{(pt)}$ using Givental's methods, the partition function $\cZ^{un}$ is composed of $32$ functions $\cZ^{(pt)}$, which means more variables than the partition functions $\cF_0^{\PP_{4,4,2}}$. A similar result was obtained in \cite{B0}, where it was shown that the Frobenius manifold of the Gromov--Witten theory of the orbifold $\PP^1_{2,2,2,2}$ is a \textit{submanifold} of a certain higher-dimensional Frobenius manifold.

It's also important to note that the $R$--matrix of Givental--Teleman theory is very hard to write explicitly, making the use of the full theory very restricitve. In contrast to this, our $R_{GW}$ is written in a closed formula.

The proof of the CY/LG correspondence is also interesting by itself, since it uses the theory of modular forms, but gives a result in terms of Givental's action. This part of the current article is closely related to the independent work of Shen and Zhou \cite{SZ2}. Their result is more systematic from the point of view of the theory of modular forms, however they don't give the particular Givental's action. Furthermore, when requiring some explitic data to be compared Shen--Zhou consider the solutions to a certain ODE fixed by the initial conditions, while we use particular values of the modular forms. This difference is also related to the different approaches to the primitive form change on the B side. Our approach also shows the holomorphicity of the FJRW theory in question.

\subsection*{Acknowledgement}
The authors are grateful to Amanda Francis for the fruitful discussions during her visit to Hannover and via subsequent emails. A.B. is also grateful to Atsushi Takahashi for sharing his ideas on CY/LG correspondence.

After having proved the CY/LG correspondence in the particular case of $(\tilde E_7,G_{max})$ by our methods, we were informed that Shen and Zhou had independently developed a systematic proof for the CY/LG correspondence for all simple elliptic singularities via the use of modular forms (\cite{SZ2}). 
We are grateful to Shen and Zhou for the email conversations and also for the sharing the draft versions of our respective texts between us.

A.B. was partially supported by the DGF grant He2287/4--1 (SISYPH).

\subsection*{Organization of the paper}
In Section~\ref{section:definitionOfFJRW} we review FJRW theory for a pair $(W,G)$ and give a system of axioms by which one can compute the basic correlators. We don't give a full definition of the virtual class of Fan--Jarvis--Ruan, but rather restrict ourselves to the situation of this paper in order to avoid some  complicated formulas unnecessary to this work. In Section~\ref{section:CohFT} we recall the definition of a cohomological field theory and give many details on the particular cases of the FJRW theory of $(\tilde E_7,G_{max})$ and Gromov--Witten theory of $\PP^1_{4,4,2}$. In Section~\ref{section: CYLG via modularity} we show the CY/LG correspondence by using modularity property of the Gromov--Witten theory of $\PP^1_{4,4,2}$---this is exactly the place, where we have an intersection with \cite{SZ2}. Section~\ref{section: GiventalsAction} is devoted to the Givental's action. We give there a particular action, which yields the CY/LG correspondence of the A models discussed. In Section~\ref{section: twisted theory} so-called ``twisted'' correlators are introduced. These give us a partition function, depending on additional parameters, which generalize genus zero FJRW theory. In fact, we recover the FJRW partition function in the limit. In Section~\ref{section: Rmatrix}, we show that the twisted correlators also recover a basic theory, as discussed above, which we call the ``untwisted'' theory. We show how to recover FJRW theory from the untwisted theory using Givental's action, and use this result to give an action similar to the $R$-matrix of the Gromov--Witten theory of $\PP_{4,4,2}$. In Appendix~\ref{section: appendix}, we have given a closed formula for $F_0^{\PP_{4,4,2}}$.

\section{Definition of FJRW theory}\label{section:definitionOfFJRW}

We first introduce FJRW theory in some generality, describing the state space and the moduli space of $W$-structures together with its virtual class. 

\subsection{State Space}
The Landau--Ginzburg A model is provided by FJRW theory. The input is a pair $(W,G)$ of a quasihomogeneous polynomial and a group, which we now describe. 

Let $W\in \CC[x_1,\dots,x_N]$ be a quasihomogeneous polynomial of degree $d$ with integer weights $w_1, \dots, w_N$ such that $\gcd(w_1, \dots, w_N)=1$. For each $1\leq k\leq N$, let $q_k=\tfrac{w_k}{d}$. The central charge of $W$ is defined to be
\[
\hat c:=\sum_{k=1}^N(1-2q_k).
\]
A polynomial is \emph{nondegenerate} if 
\begin{itemize}
\item[(i)] the weights $q_k$ are uniquely determined by $W$, and
\item[(ii)] the hypersurface defined by $W$ is non-singular in projective space. 
\end{itemize}

The maximal group of diagonal symmetries is defined as
\[
G_{max}:=\set{(\varTheta_1, \dots, \varTheta_N)\subseteq \left(\mathbb{Q}/\ZZ\right)^N \mid W(e^{2\pi \im \varTheta_1}x_1,\dots,e^{2\pi \im \varTheta_N}x_N) = W(x_1,\dots,x_N)}
\]
Note that $G_{max}$ always contains the \emph{exponential grading element} $\jw:=(q_1,\dots,q_N)$. If $W$ is nondegenerate, $G_{max}$ is finite. 

\begin{remark}
One can define FJRW theory more generally for \emph{admissible} subgroups $G\subset G_{max}$ (see \cite{FJR}), but in the current work we consider only $G=G_{max}$. From now on, we will denote $G_{max}$ simply by $G$. 
\end{remark}

FJRW theory defines a state space and a moduli space of $W$-curves, from which one obtains certains numbers---called correlators---as integrals over the moduli space. Let us first fix some notation. For $h\in G$, let $\cFix(h)$ denote the fixed locus of $\CC^N$ with respect to $h$, let $N_h$ denote the dimension of $\cFix(h)$ and let $W_h$ denote $W|_{\cFix(h)}$. 
Let $W_h^{+\infty}:=(\mathrm{Re} W_h)^{-1}(\rho, \infty)$, for $\rho \gg 0$, be the so-called Milnor fiber of $W_h$.

Define
\begin{equation}\label{e:Hh}
\sH_{h}:=H^{N_h}(\cFix(h), W_h^{+\infty}; \CC)^{G},
\end{equation}
that is, $G$-invariant elements of the the middle dimensional relative cohomology of $\cFix(h)$. The state space is the direct sum of the ``sectors'' $\sH_h$, i.e. 
\[
\sH_{W,G} := \bigoplus_{h\in G}\sH_h.
\]

Let $G^{nar}=\set{h\in G\mid N_h=0}$. These summands $\sH_h$ for $h\in G^{nar}$ are the so-called \emph{narrow} sectors. 

$\sH_{W,G}$ is $\mathbb{Q}$-graded by the $W$-degree. To define this grading, first note that each element $h\in G$ can be uniquely expressed as a tuple of rational numbers
\[
h=(\Theta_1^h,\dots,\varTheta_N^h)
\]
with $0\leq \varTheta_k^h < 1$.

We first define the \emph{degree-shifting number} 
\[
\iota(h):=\sum_{k=1}^N(\varTheta_k^h-q_k).
\]
For $\alpha_h \in \mathscr{H}_{h}$, the (real) $W$-degree of $\alpha_h$ is defined by
\begin{equation}\label{e:degw}
\deg_W(\alpha_h):=N_h+2\iota(h).
\end{equation}

The sector indexed by $\jw$, is one--dimensional, and has $W$--degree 0. 
This sector is unique with this property. 

Because $\cFix(h)= \cFix(h^{-1})$ there is a nondegenerate pairing
\[
\br{-,-}:\sH_h\times\sH_{h^{-1}}\to \CC,
\]
the residue pairing of $W_h$, which induces a symmetric nondegenerate pairing
\[
\br{- ,- }:\sH_{W,G}\times \sH_{W,G}\to \CC.
\]


\subsection{Moduli of W-curves}

Recall that an \emph{$n$-pointed orbifold curve} is a stack of Deligne--Mumford type with at worst nodal singularities with orbifold structure only at the marked points and the nodes. We require the nodes to be \emph{balanced}, in the sense that the action of the generator of the stabilizer group $\ZZ_k$ be given by 
\[
(x,y)\mapsto (e^{2 \pi \im/ k} x,e^{-2 \pi \im/k}  y).
\]
Given such a curve $\cC$, let $\omega$ be its dualizing sheaf. The \emph{log-canonical bundle} is 
\[
\omega_{\log}:=\omega(p_1+ \dots + p_n)
\]

In what follows, we will assume $d$, the degree of $W$, is also the exponent of $G_{max}$, i.e. for each $h\in G_{max}$, $h^d=\id$. This is not the case in general, but it simplifies the exposition, while still giving a general enough picture. 

The FJRW correlators were first defined in \cite{FJR}, but we will follow a slightly different treatment as given in \cite{ChR}, where it is also shown that the two definitions agree. The reason for our choice, is that \cite{ChR} allows us to use Givental's formalism to determine the FJRW correlators. 

A $d$-stable curve is a proper connected orbifold curve $\cC$ of genus $g$ with $n$ distinct smooth markings $p_1,\dots,p_n$ such that
\begin{enumerate}
\item[(i)] the  $n$-pointed underlying coarse curve is stable, and
\item[(ii)] all the stabilizers at nodes and markings have order $d$. 
\end{enumerate}

The moduli stack $\overline{\cM}_{g,n,d}$ parametrizing such curves is proper, smooth and has dimension $3g-3+n$. It differs from the moduli space of curves only because of the stabilizers over the normal crossings (see \cite{ChR}).  

We can write $W$ as a sum of monomials $W=W_1+\dots+W_s$, where $\displaystyle W_i=c_i\prod_{k=1}^N x_k^{a_{ik}}$ with $a_{ik}\in \NN$ and $c_i\in \CC$. Given line bundles $\cL_1, \ldots , \cL_N$ on the $d$-stable curve $\cC$, we define the line bundle \[W_i(\cL_1,\dots,\cL_N) := \bigotimes_{k=1}^N\cL_k^{\otimes a_{ik}}.\] 

\begin{definition}A \emph{$W$-structure} is comprised of the data 
\[
(\cC, p_1,\dots, p_n, \cL_1,\dots,\cL_N,\varphi_1,\dots \varphi_N),
\]
where $\cC$ is an $n$-pointed $d$-stable curve, the $\cL_k$ are line bundles on $\cC$ satisfying 
\[
W_i(\cL_1,\dots,\cL_N)\cong \omega_{\log},
\] 
and for each $k$, $\varphi_k:\cL_k^{\otimes d}\to \omega_{\log}^{w_k}$ is an isomorphism of line bundles. 
\end{definition}
There exists a moduli stack of $W$-structures, denoted by $\cW_{g,n}$ (see \cite{FJR2}, \cite{ChR} for the construction). \footnote{This defninition differs slightly from \cite{FJR2}, in that only the isomorphisms $\phi_k$ are part of the data. But it is shown in \cite{ChR} that the defintions agree.}

\begin{proposition}[\cite{ChR}]
The stack $\cW_{g,n}$ is nonempty if and only if $n>0$ or $2g-2$ is a positive multiple of $d$. It is a proper, smooth Deligne--Mumford stack of dimension $3g-3+n$. It is etale over $\overline{\cM}_{g,n,d}$ of degree $|G_{max}|^{2g-1+n}/d^N$. 
\end{proposition}

Let $\bv{h}=(h_1,\dots,h_n)$, with $h_i=(\varTheta_1^i, \dots, \varTheta_N^i)$. Define $\cW_{g,n}(\bv{h})$ to be the stack of $n$-pointed, genus $g$ $W$-curves for which the generator of the isotropy group at $p_j$ acts on $\cL_k$ by multiplication by $e^{2 \pi \im \ \varTheta_k^j}$. We can write $\varTheta_k^j=m_k^j/d$ for some integer $0\leq m_k^j <d$, which we call the multiplicity of  $\cL_k$ at $p_j$ and we denote by $\mult_{p_i}\cL_k$. The following proposition describes a decomposition of $\cW_{g,n}$ in terms of multiplicities:
\begin{proposition}[\cite{ChR,FJR}]\label{p:linebundledegree}
The stack $\cW_{g,n}$ can be expressed as the disjoint union
\[
\cW_{g,n}=\coprod \cW_{g,n}(\bv{h})
\]
with each $\cW_{g,n}(\bv{h})$ an open and closed substack of $\cW_{g,n}$. Furthermore, $\cW_{g,n}(\bv{h})$ is non-empty if and only if
\begin{align*}
h_i &\in G_{max}, \: i=1,\dots, n\\
q_k(2g-2+n)-\sum_{i=1}^n \varTheta_k^{i} &\in \ZZ, \: \: \:\: \: \: \: k=1,\dots, N.
\end{align*}
\end{proposition}
The second condition comes from the pushforward of $\cL_k$ to the course underlying curve, which must have integer degree. We denote the universal curve by $\pi:\sC\to \cW_{g,n,G}(\bv{h})$ and the universal $W$-structure by $(\LL_1,\dots, \LL_N)$.

\subsection{Axioms of FJRW theory}\label{ss:FJRWaxioms}
For each substack $\cW_{g,n}(\bv{h})$, one may define a virtual cycle (see \cite{FJR}, \cite{FJR2}) 
\[
[\cW_{g,n}(\bv{h})]^{vir}\in H_*(\cW_{g,n}(\bv{h}), \QQ)\otimes \prod_{i=1}^n H_{N_{h_i}}(\cFix(h_i), W_{h_i}^{+\infty}; \CC)^{G}  
\]
which satisfies the following axioms: 
\begin{fjraxiom}[Degree]\label{a:fjr1} The virtal cycle has degree 
\[
2\left((\hat c -3)(1-g)+ n -\sum_{i=1}^n \iota(h_i)\right).
\]
In particular, if this number is not an integer, then $[\cW_{g,n}(\bv{h})]^{vir}=0$.
\end{fjraxiom}

\begin{fjraxiom}[Line bundle degree]\label{a:line bundle}
The degree of the pushforward $|\cL_k|$
\[
q_k(2g-2+n)-\sum_{i=1}^n \varTheta_k^{h_i}
\]
must be an integer (as in Proposition \ref{p:linebundledegree}), otherwise $[\cW_{g,n}(\bv{h})]^{vir}=0$.
\end{fjraxiom}

\begin{fjraxiom}[Symmetric Group invariance] For any $\sigma \in S_n$, we have 
\[[\cW_{g,n}(h_1,\dots,h_n)]^{vir}=[\cW_{g,n}(h_{\sigma(1)}, \dots ,h_{\sigma(n)})]^{vir}.\]
\end{fjraxiom}

\begin{fjraxiom}[Deformation invariance] Let $W_t\in \CC[x_1,\dots,x_N]$ be a family of nondegenerate quasihomogeneous polynomials depending smoothly on a real parameter $t\in [a,b]$. Suppose that $G$ is the common isomorphism group of $W_t$. The corresponding moduli of $W$-structures are naturally isomorphic, and the virtual cycle $[\cW_{g,n}(h_1,\dots,h_n)]^{vir}$ associated to $(W_t,G)$ is independent of $t$. 
\end{fjraxiom}

\begin{fjraxiom}[$G_{max}$-invariance] There is a natural $G_{max}$ action on $H_*(\cW_{g,n}(\bv{h}), \QQ)$ and $H_{N_{h_i}}(\cFix(h_i), W_{h_i}^{+\infty}; \CC)^{G}$. The virtual cycle $[\cW_{g,n}(\bv{h})]^{vir}$ is invariant under the induced $G_{max}$ action on the tensor product.
\end{fjraxiom}

\begin{fjraxiom}[Concavity]\label{a:fjrlast}
Suppose that $h_i\in G^{nar}$ for all $i$. If $\pi_*\left(\bigoplus_{k=1}^N \LL_k \right)=0$, then the virtual class is given by
\[
[\cW_{g,n}(\bv{h})]^{vir}= c_{top}\left(\big(R^1\pi_*\bigoplus_{k=1}^N \LL_k \big)^\vee\right)\cap [\cW_{g,n}(\bv{h})]
\]
and the substack $\cW_{g,n}(\bv{h})$ is called \textit{concave}.
\end{fjraxiom}

\begin{remark}
This last axiom can also be modified to take into account the restriction to boundary components on $\cW_{g,n}$, i.e. $W$-curves with reducible underlying curve (cf. \cite{FJR}). 
\end{remark}

\begin{remark}
Some authors also include the ``Index Zero`` axiom, but in full generality, both concavity and index zero are actually a part of a larger axiom involving the topological Euler class and the Witten map, but we will not need the full statement here. We also do not include the sums of singularities axiom.
\end{remark}

There are a few other axioms that are satisfied by $[\cW_{g,n}(\bv{h})]^{vir}$ that are more complicated to state (cf. \cite{FJR}), so we will not list them here. They show, for example, that the virtual class behaves well with respect to cutting along nodes, ensuring that FJRW theory defines a cohomological field theory, as we will see. 

The stacks $\cW_{g,n}$ are also equipped with $\psi$-classes, which are pulled back from the course underlying curve.

\section{Cohomological field theories on $\M{g}{n}$}\label{section:CohFT}

We briefly recall some basic facts about cohomological field theories as introduced in \cite{KM}.
  
\subsection{Cohomological Field Theory axioms}\label{ss:cohft axioms}
Let $(V,\eta)$ be a finite-dimensional vector space with a nondegenerate pairing.
Consider a system of linear maps 
\[
\Lambda_{g,n}: V^{\otimes n} \rightarrow H^*(\M{g}{n}),
\] 
defined for all $g,n$ such that $\M{g}{n}$ exists and is non-empty.
The set $\Lambda_{g,n}$ is called a \emph{cohomological field theory on $(V,\eta)$}, or CohFT, if it satisfies the following axioms:
    
\begin{cohftaxiom}[$S_n$ invariance] $\Lambda_{g,n}$ is equivariant with respect to the $S_n$-action permuting the factors in the tensor product and the numbering of marked points in $\M{g}{n}$.
\end{cohftaxiom}

\begin{cohftaxiom}[Cutting trees] For the gluing morphism $\rho: \M{g_1}{n_1+1} \times \M{g_2}{n_2+1} \rightarrow \M{g_1+g_2}{n_1+n_2}$ we have:
\[
\rho^* \Lambda_{g_1+g_2,n_1+n_2} = (\Lambda_{g_1, n_1+1} \cdot \Lambda_{g_2,n_2+1}, \eta^{-1}),
\]
where we contract with $\eta^{-1}$ the factors of $V$ that correspond to the node in the preimage of $\rho$.
\end{cohftaxiom}
      
\begin{cohftaxiom}[Cutting loops] For the gluing morphism $\sigma: \M{g}{n+2} \rightarrow \M{g+1}{n}$ we have:
\[
\sigma^* \Lambda_{g+1,n} = (\Lambda_{g, n+2}, \eta^{-1}),
\]
where we contract with $\eta^{-1}$ the factors of $V$ that correspond to the node in the preimage of $\sigma$.
\end{cohftaxiom}
    
In this paper we further assume that the CohFT $\Lambda_{g,n}$ is \emph{unital}---i.e. there is a fixed vector $\textbf 1 \in V$ called the \emph{unit} such that the following axioms are satisfied.
  
\begin{unitaxiom} For every $\alpha_1,\alpha_2 \in V$ we have: $\eta(\alpha_1, \alpha_2) = \Lambda_{0,3}(\textbf{1} \otimes \alpha_1 \otimes \alpha_2)$.
\end{unitaxiom}

\begin{unitaxiom} Let $\pi: \M{g}{n+1} \rightarrow \M{g}{n}$ be the map forgetting the last marking, then:
\[
\pi^* \Lambda_{g,n}(\alpha_1 \otimes \dots \otimes \alpha_n) = \Lambda_{g,n+1}(\alpha_1 \otimes \dots \otimes \alpha_n \otimes \textbf{1}).
\]
\end{unitaxiom}

Another important property of CohFTs is the notion of quasihomogeneity. A CohFT $\Lambda_{g,n}$ on $(V,\eta)$ is called \emph{quasihomogeneous} if the vector space $V$ is graded by $\deg: V \to \QQ$ and there is a number $\hat c$, such that for any ${\alpha_1, \dots, \alpha_n \in V}$
\[
\sum_{i=1}^n \deg(\alpha_i) = \hat c +n + g - 3
\]
whenever $\langle \alpha_1, \dots, \alpha_n \rangle_{g,n} \neq 0$. The number $\hat c$ is called the \emph{central charge}.

\begin{remark}
The space of all quasihomogeneous CohFTs is discrete in the space of all CohFTs. However these CohFTs posess several properties that make them easier to work with. The CohFTs we are going to work with in this text are quasihomogeneous. In FJRW theory the state space is graded by $W$-degree, and in GW theory, the state space is graded simply by the cohomological degree. 
\end{remark}

In what follows we will denote the CohFT just by $\Lambda$ rather than $\Lambda_{g,n}$ when there is no ambiguity.

Let $\psi_i \in H^*(\M{g}{n})$,  $1 \le i \le n$ be the so-called $\psi$-classes. The genus $g$, $n$-point correlators of the CohFT are the following numbers:
\[
\langle \tau_{a_1}(e_{\alpha_1}) \dots \tau_{a_n}(e_{\alpha_n}) \rangle_{g,n}^\Lambda := \int_{\M{g}{n}} \Lambda_{g,n}(e_{\alpha_1} \otimes \dots \otimes e_{\alpha_n}) \psi_1^{a_1} \dots \psi_n^{a_n}.
\]

Denote by $\cF_g$ the generating function of the genus $g$ correlators, called genus $g$ potential of the CohFT:
\[
\cF_g := \sum \frac{\langle \tau_{a_1}(e_{\alpha_1}) \dots \tau_{a_n}(e_{\alpha_n}) \rangle_{g,n}^\Lambda}{\Aut( \{ \boldsymbol \alpha, \bf a \})} \ t^{a_1,\alpha_1} \dots t^{a_n,\alpha_n}.
\]

It is useful to assemble the correlators into a generating function called partition function of the CohFT\footnote{One could consider a family of partition functions $\cZ_\tau$ for $\tau \in V$ by shifting the variables. We will explain it later in Section~\ref{section: GiventalsAction}.}:
\[
\cZ := \exp \left( \sum\nolimits_{g \ge 0} \hbar^{g-1} \mathcal F_g \right).
\]

We will also make use of the so-called \textit{primary} genus $g$ potential that is a function of the variables $t^\alpha:=t^{0,\alpha}$ defined as follows:
\[
  F_g := \cF_g \mid_{t^\alpha := t^{0,\alpha}, \ t^{\ell,\alpha} = 0, \forall \ell \ge 1}
\]
what is also sometimes called a \textit{restriction to the small phase space}.

\subsection{CohFT of FJRW theory and Gromov--Witten theory}
In the space of all cohomological field theories there are certain special theories, called also sometimes ``geometric'' since they correspond to some geometry. These include FJRW theory and GW theory. 
  
Consider the FJRW theory of a pair $(W,G)$. Its moduli space of $W$-structures has a good virtual cycle $[\cW_{g,n}(\bv{h})]^{vir}$ as it was explained in Section~\ref{ss:FJRWaxioms}.

However, we can also push forward to $\overline{\cM}_{g,n}$ via the map $s:\cW_{g,n}\to \overline{\cM}_{g,n}$. Let $\alpha_i \in \sH_{h_i}$, and $\boldsymbol{\alpha}=(\alpha_1,\dots,\alpha_n)$. We define 
\begin{equation}\label{e:fjrcohft}
\Lambda_{g,n}^{FJRW}(\boldsymbol{\alpha}) =\frac{|G|^g}{\deg s}PD s_*\Big([\cW_{g,n}(\bv{h})]^{vir}\cap \prod_{i=1}^n \alpha_i\Bigg). 
\end{equation}
Here $PD$ denotes the Poincare dual. 

\begin{theorem}[Theorem~4.2.2 in \cite{FJR}]
For any admissible pair $(W,G)$ the system of maps $\Lambda_{g,n}^{FJRW}$ defines a unital CohFT on the vector space $\sH_{W,G}$.
\end{theorem}

In what follows we denote simply by $\cZ^{(W,G)}$ and $\cF_g^{(W,G)}$ the partition function and genus $g$ potential of the CohFT above for a fixed admissible pair $(W,G)$. As a consequence of the properties of the virtual cycle $[\cW_{g,n}]^{vir}$, these functions also satisfy certain additional properties in addition to those common to all CohFTs.

The second important class of the CohFTs is given by the Gromov--Witten theories. We recall very briefly the definition and refer to \cite{A} for a full exposition. Let $\cX$ be an orbifold and $\beta \in H_2(\cX, \ZZ)$. There is the moduli stack $\M{g}{n}(\cX, \beta)$ of degree $\beta$ stable orbifold maps from the genus $g$ curve with $n$ marked points to $\cX$. The \emph{orbifold cohomology} $H^*_{orb}(\cX)$, with the nondegenerate pairing, serves as a state space in this theory. Similarly to the FJRW theory there is a good virtual cycle $[\M{g}{n}(\cX,\beta)]^{vir}$ so that one could consider the Gromov--Witten invariants given by the intersection theory on $\M{g}{n}(\cX,\beta)$.

Again, by considering a push forward $s: \M{g}{n}(\cX,\beta) \to \M{g}{n}$ we get a CohFT associated to $\cX$ with the fixed $\beta$.
\[
\Lambda^{GW}_{g,n,\beta} := \frac{1}{\deg s} PD s_*\Big([\M{g}{n}(\cX,\beta)]^{vir}\cap \prod_{i=1}^n ev^*_i(\alpha_i)\Bigg). 
\]
  
The fact that this map defines a CohFT follows from a more general statement and could be found for example in \cite{A}.

As with FJRW theory, the CohFT obtained satisfies some additional properties. One of the most important for us is the so-called \emph{divisor equation}. When $H_2(\cX,\ZZ)$ is one-dimensional, as is the case in this work, it allows us to sum over all classes $\beta$ and obtain a CohFT, $\Lambda_{g,n}^{GW}$ depending on $\cX$ only.  We denote by $\cZ^{\cX}$ and $\cF_g^{\cX}$ the partition function and genus $g$ potential of the CohFT $\Lambda_{g,n}^{GW}$.

\subsection{Reconstruction in genus zero}\label{subsection: reconstruction in genus zero}
It is often useful to be able to express all correlators of a given CohFT from some finite list. This is usually referred to as \emph{reconstruction}.
  
Due to the topology of the space $\M{0}{n}$ the small phase space potential of a CohFT on $(V,\eta)$ satisfies the so-called \emph{WDVV equations}. For any four fixed $0 \le i,j,k,l \le \dim(V)-1$ it reads:
\[
\sum_{p,q} \frac{\partial^3 F_0}{\partial t^i \partial t^j \partial ^p} \eta^{p,q} \frac{\partial^3 F_0}{\partial t^q \partial t^k \partial^l} = \sum_{p,q} \frac{\partial^3 F_0}{\partial t^i \partial t^k \partial ^p} \eta^{p,q} \frac{\partial^3 F_0}{\partial t^q \partial t^j \partial^l}.
\]

It follows from here that the genus zero three-point correlators endow $V$ with the structure of an associative and commutative algebra by setting $e_i \circ e_j := \sum_{p,k} \langle e_i, e_j, e_p \rangle_{0,3} \cdot \eta^{p,k} \cdot e_k$.
  
\begin{definition}
The vector $\gamma \in V$ is called \emph{primitive} if there is no $\gamma_1,\gamma_2 \in V$, such that  $\gamma = \gamma_1 \circ \gamma_2$ and $0 < \deg(\gamma_1) \le \deg(\gamma_2) < \deg(\gamma)$. We call a correlator $\langle \dots \rangle_{0,n}$ \emph{basic} if it involves at most two non-primitive insertions.
\end{definition}
  
The following lemma will be used later on.
  
\begin{lemma}[Lemma~6.2.8 
in \cite{FJR}]\label{lemma: reconstruction in genus zero}
Fix a quasihomogeneous CohFT on $(V, \eta)$.

If $\deg(\alpha) \le \hat c$ for all vectors $\alpha \in V$ then all the genus zero correlators are uniquely determined by $\eta$ and the $n$-point genus zero correlators with:
\[
n \le 2 + \frac{1 + \hat c}{1 - P}, \quad P:= \max_{\substack{v \in V \\ v \text{ is primitive}}} \deg(v).
\]
\end{lemma}

The proof of this lemma is based on the analysis of WDVV equation of a quasihomogeneous CohFT.

\subsection{FJRW Correlators for $\tilde E_7$}\label{subsection: E_7 state space}
In this article, we consider the polynomial $W=x^4+y^4+z^2$ defining the $\tilde E_7$ singularity. The polynomial $W$ is quasihomogeneous with weights $q_1=\tfrac 14$, $q_2=\tfrac 14$ and $q_3=\tfrac 12$. The group $G=G_{max}$ is generated by the elements $\rho_1 := (q_1,0,0)$, $\rho_2 := (0,q_2,0)$, and $\rho_3 := (0,0,q_3)$. In this description $\jw_W = \rho_1 \rho_2 \rho_3$. 

Working through the definition, one sees that in this case only narrow group elements contribute to the state space. Furthermore, via \eqref{e:Hh}, we see that $\cH_h$ is one--dimensional when $h\in G^{nar}$. In this case we denote the fundamental class in $\sH_h$ by $\phi_h$. 
With $\phi_h$ defined as above, the set $\set{\phi_h}_{h\in G^{nar}}$ defines a basis of $\sH_{W,G}$, i.e. we have  
\[
\sH_{W,G} := \bigoplus_{\substack{1 \le a \le 3 \\ 1 \le b \le 3}}\CC\cdot\phi_{\rho_1^a \rho_2^b \rho_3}.
\]
The pairing is determined by the following values on this basis:
\[
\br{\phi_{h_1},\phi_{h_2}}:=\begin{cases}
    1 & \text{ if } h_1=(h_2)^{-1}\\
    0 & \text{otherwise}.
    \end{cases}
\]

In the following lemma, we show that the entire FJRW theory is concave, namely all substacks $\cW_{g,n}(\bv{h})$ satisfy the concavity condition. So by the concavity axiom, we can replace the virtual class by the fundamental class capped with the top chern class of a line bundle. 

\begin{proposition}\label{p:concave}
The genus zero FJRW theory for $(\tilde E_7,G_{max})$ is concave.
\end{proposition}

\begin{proof} The proof has been given in several places, including \cite{ChR}, \cite{ChIR}, \cite{PrSh}, so we will not give it in detail here. It consists of checking that over any geometric point $(\cC, p_1, \dots, p_n,$ $\cL_1, \cL_2, \cL_3, \varphi_1, \varphi_2, \varphi_3)$ in the moduli space,  $\bigoplus_{k=1}^3 H^0(\cC,\cL_k) = 0$. This is done by checking the degree of the line bundle satisfies for each connected component $C_v$ of $\cC$
\[
\deg(|\cL_k|_{C_v}) \leq  q_k(\#\text{nodes}(C_v)-2)<\#\text{nodes}(C_v)-1.
\] 
\end{proof}

The genus $0$ potential of the FJRW theory is written in the variables $\tilde t_{ab}$, for $1\leq a\leq 3$ and $1\leq b\leq 3$, corresponding to the vectors $\phi_{\rho_1^a\rho_2^b\rho_3}$. It is always the case that $\phi_{\jw}$ is the unit. Thus the variable $\tilde t_{11}$ corresponds to the unit of the CohFT for $(\tilde{E}_7,G_{max})$.

Using axioms FJR~\ref{a:fjr1}--FJR~\ref{a:fjrlast} (the data is also in \cite[Section~3.3]{MS1}), we get the following expression for the genus zero small phase space potential of $(\tilde E_7, G_{max})$:
\begin{align*}
  F_0^{\tilde E_7, G_{max}} & = \frac{\tilde t_{11}^2 \tilde t_{33}}{2} + \tilde t_{11}\left(\tilde t_{21}\tilde t_{23}+\tilde t_{12}\tilde t_{32}+\tilde t_{13}\tilde t_{31}+\frac{\tilde t_{22}^2}{2}\right) + \tilde t_{12}\tilde t_{21}\tilde t_{22}
  \\ 
  & +\frac{\tilde t_{12}^2\tilde t_{31}}{2} +\frac{\tilde t_{21}^2\tilde t_{13}}{2} - \tilde t_{33}\left(\frac{\tilde t_{21}^2\tilde t_{31}}{8} + \frac{\tilde t_{12}^2\tilde t_{13}}{8}\right) + O(\tilde \bt_{+}^4,\tilde t_{33}),
\end{align*}
where $\tilde \bt_+$ is the set of all coordinates except $\tilde t_{33}$.

We can rephrase Lemma~\ref{lemma: reconstruction in genus zero} for this case in the following lemma. 

\begin{lemma}[Lemma~3.6 in \cite{MS1} and Theorem~3.4 in \cite{KS}]\label{lemma:reconstruction of FJR}
  Using the WDVV equation, all genus $0$ primary correlators of FJRW theory $(\tilde E_7,G_{max})$ are uniquely determined by the FJRW algebra and the basic $4$-point correlators that have exactly one insertion of $\rho_1^3\rho_2^3\rho_3$.
\end{lemma}

\begin{remark}
  It follows immediately from Lemma~\ref{lemma:reconstruction of FJR} and proof of Lemma~\ref{lemma: reconstruction in genus zero} that the genus $0$ potential $F_0^{\tilde E_7, G_{max}} \in \QQ[[t]]$ because all the ``primary'' data is rational and the WDVV equation doesn't involve anything non-rational.
\end{remark}

\subsection{Gromov--Witten theory of elliptic orbifolds}\label{section: GW of elliptic orbifolds} Let the ordered set $(a_1,a_2,a_3)$ be either $(3,3,3)$, $(4,4,2)$ or $(6,3,2)$. Consider $\cX := \PP^1_{a_1,a_2,a_3}$, one of the so-called \emph{elliptic orbifolds}. They can be either viewed as the projective line with the three isotropy points of the order $a_1$,$a_2$,$a_3$ respectively, or as the quotients of an elliptic curve by a group of order $3$, $4$ and $6$ respectively. 
In this case we have 
\[
\dim (H^*_{orb}(\PP^1_{a_1,a_2,a_3})) = 2 + \sum_{i=1}^3 \left( a_i - 1 \right).
\] 
The space $H^*_{orb}(\PP^1_{a_1,a_2,a_3})$ has the generators: 
\[
\Delta_0, \Delta_{-1}, \quad \Delta_{i,j}, \quad 1 \le i \le 3, \ 1 \le j \le a_i-1,
\]
so that $H^*_{orb}(\PP^1_{a_1,a_2,a_3}) \simeq \QQ \Delta_0 \oplus \QQ \Delta_{-1} \bigoplus_{i=1}^3\bigoplus_{j=1}^{a_i-1} \QQ \Delta_{i,j}$, and 
$H^0(\PP^1_{a_1,a_2,a_3},\QQ) \simeq \QQ \Delta_0$, $ H^2(\PP^1_{a_1,a_2,a_3},\QQ) \simeq \QQ \Delta_{-1}$, 

The pairing is given by:
\[
\eta(\Delta_0,\Delta_{-1}) = 1, \quad \eta(\Delta_{i,j}, \Delta_{k,l}) = \frac{1}{a_i} \delta_{i,k} \delta_{j+l,a_i}.
\]
The potential of this CohFT is then written in the coordinates $t_0,t_{-1}$ and $t_{i,j}$, corresponding to the classes $\Delta_0, \Delta_{-1}$ and $\Delta_{i,j}$ respectively.

As with FJRW theory, it turns out that one needs to know only certain finite list of the correlators in order to compute all the correlators of these GW theories. Such correlators were found explicitly by \cite{SZ1} and used independently by the first author to write down the genus $0$ potentials explicitly. In what follows we will be particularly interested in the GW theory of $\PP^1_{4,4,2}$. We give explicitly the genus $0$ potential of this orbifold in the appendix.

\subsection{GW theory of $\PP^1_{4,4,2}$}\label{section: GW of X4}
Let $\vartheta_2(q)$, $\vartheta_3(q)$, $\vartheta_4(q)$ be the following infinite series in a formal variable $q$:
\begin{align*}
    \vartheta_2(q) &= 2 \sum_{k=0}^\infty q^{\frac{1}{2} \left(k + \frac{1}{2}\right)^2},
    \quad
    \vartheta_3(q) = 1 + 2 \sum_{k=1}^\infty q^{\frac{k^2}{2}},
    \\
    &\vartheta_4(q) = 1+2 \sum_{k=1}^\infty (-1)^kq^{\frac{k^2}{2}},
\end{align*}
and also
\[
f(q) := 1 - 24 \sum_{k=1}^\infty \dfrac{k q^k}{1-q^k}.
\]
These series are the $q$-expansions of the Jacobi theta constants and the second Eisenstein series respectively. However at the moment we consider them only as the formal series in the variable $q$.
  
Consider the functions $x(q)$,$y(q)$,$z(q)$,$w(q)$ defined as follows:
\begin{align*}
    x(q) &:= \left( \theta_3(q^8) \right)^2, \ y(q) := \left( \theta_2(q^8) \right)^2, \ z(q) := \left( \theta_2(q^4) \right)^2,   \\ 
    w(q) &:= \frac{1}{3} \left( f(q^4) - 2 f(q^8) + 4 f(q^{16}) \right).
\end{align*}

In what follows the function $z(q)$ will be sometimes skipped because the following identity holds:
\[
z(q)^2 = 4 x(q) y(q).
\]

\begin{proposition}\label{prop: gw potential basics}
The potential $F^{\PP^1_{4,4,2}}_0$ has an explicit form via the functions defined above. Namely:
\[
F^{\PP^1_{4,4,2}}_0 \in \QQ \left[t_0,t_{-1},t_{i,j}, x,y,z,w \right],
\]
where $x = x(q)$, $y = y(q)$, $z = z(q)$ and $w = w(q)$ as above. Moreover it satisfies the following homogeneity property:
\[
F^{\PP^1_{4,4,2}}_0 \left( t_0,t_{-1},t_{i,j}, x,y,z,w \right) = \alpha^{-2}F^{\PP^1_{4,4,2}}_0 \left( t_0,t_{-1}, \alpha \cdot t_{i,j}, \frac{x}{\alpha},\frac{y}{\alpha},\frac{z}{\alpha},\frac{w}{\alpha^2} \right),
\]
for any $\alpha \in \CC^*$.
\end{proposition}

\begin{proof}
This is clear from the explicit form of the potential---see Appendix~\ref{section: appendix}.
\end{proof}
  
Up to the $4$-th order terms in $t_{i,k}$ we have:
\[
\begin{aligned}
F_0^{\PP^1_{4,4,2}} = & \frac{1}{2} t_0^2 t_{-1} + t_0\left(\frac{1}{4} t_{1,1} t_{1,3}+\frac{1}{8}t_{1,2}^2+\frac{1}{4} t_{2,1} t_{2,3}+\frac{1}{8}t_{2,2}^2+\frac{1}{4}t_{3,1}^2\right) \\ 
     & +\frac{1}{8} x(q)\left(t_{1,1}^2 t_{1,2}+t_{2,1}^2 t_{2,2}\right)+\frac{1}{8} y(q) \left(t_{1,2} t_{2,1}^2+t_{1,1}^2 t_{2,2}\right) \\
      & +\frac{1}{4} z(q) t_{1,1} t_{2,1} t_{3,1} + O(t_{i,k}^4,t_{-1}),
\end{aligned}
\]
where $q = \exp(t_{-1})$.

Considering also the change of the variables $t_{-1} = t_{-1}(\tau) =\tfrac{2 \pi \im \tau}{4}$ we can consider the functions $x(q(\tau))$, $y(q(\tau))$, $z(q(\tau))$ as modular forms and $w(q(\tau))$ as a quasi-modular form. This means in particular that these functions have a large domain of holomorphicity and satisfy certain modularity condition. The first property holds also by the primary potential $F_0^{\PP^1_{4,4,2}}$, however the second---modularity---is slightly more complicated. It was shown in \cite{B2} that primary potentials of all elliptic orbifolds satisfy the modularity property, too.

\section{CY/LG correspondence via modularity}\label{section: CYLG via modularity}
Consider a unital CohFT $\Lambda$ on $(V,\eta)$ with unit $e_0$. Let $\{e_0,\dots,e_m\}$ be the basis of $V$, such that $\eta_{0,k} = \delta_{k,m}$. Define the coordinates $t_0,\dots t_m$ corresponding to this basis. Due to Axiom~(U1) of a unital CohFT, the primary genus zero potential of $\Lambda$  reads in coordinates:
\[
    F_0(t_0,\dots,t_m) = \frac{t_0^2t_m}{2} + t_0 \sum_{0 < \alpha \le \beta < m} \eta_{\alpha,\beta} \frac{t_\alpha t_\beta}{|\mathrm{Aut}(\alpha,\beta)|} + H(t_1,\dots,t_m),
\]
where $H$ is a function, not depending on $t_0$.

For any 
$A = \begin{pmatrix} 
  a & b \\ c & d
  \end{pmatrix} \in \mathrm{SL}(2, \mathbb{C})$ consider another function $F_0^A = F_0^A(t_0,\dots,t_m)$ defined by:
\begin{equation}\label{e:A Action}
\begin{aligned}
F_0^A  & := \frac{t_0^2t_m}{2} + t_0 \sum_{0 < \alpha \le \beta < m} \eta_{\alpha,\beta} \frac{t_\alpha t_\beta}{|\mathrm{Aut}(\alpha,\beta)|} + \frac{c\left( \sum_{0 < \alpha \le \beta < m} \eta_{\alpha,\beta} \frac{t_\alpha t_\beta}{|\mathrm{Aut}(\alpha,\beta)|} \right)^2}{2(ct_m+d)}\\
      & + (ct_m +d)^2 H \left(\frac{t_1}{ct_m + d},\dots,\frac{t_{m-1}}{ct_m + d}, \frac{at_m + b}{ct_m + d} \right).
\end{aligned}
\end{equation}

It is not hard to see that $F_0^A$ is solution to WDVV equation. One could also give a CohFT, whose genus $0$ primary potential it is (see \cite{B1} for details). We also write $A \cdot F_0 := \left( F_0 \right)^A$. Consider the action of a particular matrix $A^{CY/LG}$:
\[
A^{CY/LG} := 
 \begin{pmatrix}
	\dfrac{1}{2\Theta} & -\dfrac{\pi \Theta}{2}\\
	\dfrac{1}{\pi \Theta} & \Theta
 \end{pmatrix},
    \quad \Theta := \frac{\sqrt{2\pi}}{\left(\Gamma \left( \frac{3}{4} \right)\right)^2}.
\]
Main statement of this section is the following theorem.

\begin{theorem}\label{theorem: CYLG}
Let $F_0^{(\tilde{E}_7,G_{max})}$ and $F_0^{\PP_{4,4,2}}$ be the primary genus $0$ potentials of the FJRW theory of $(\tilde E_7,G_{max})$  and GW theory of $\PP^1_{4,4,2}$ respectively.
Then we have:
\[
F_0^{(\tilde{E}_7,G_{max})}(\tilde \bt) = A^{CY/LG} \cdot F_0^{\PP_{4,4,2}}(\bt),
\]
where $\tilde \bt = \tilde \bt(\bt)$ is the following linear change of the variables:
\begin{align*}
  & t_{1,1} = \im\sqrt{2}\left(\tilde t_{12}-\tilde t_{21}\right), \ t_{1,2} = -\tilde t_{13}+\sqrt{2} \tilde t_{22}-\tilde t_{31}, \ t_{1,3} = \im\sqrt{2}\left(\tilde t_{23}-\tilde t_{32}\right), \\ 
    & t_{2,1} = \sqrt{2} \left(\tilde t_{12}+\tilde t_{21}\right), \ t_{2,2} = \tilde t_{13} + \sqrt{2} \tilde t_{22}+\tilde t_{31}, \ t_{2,3} = \sqrt{2} \left(\tilde t_{23}+\tilde t_{32}\right), \\
    & t_{3,1} = \im\left(\tilde t_{13}-\tilde t_{31}\right), \\ 
    & t_0 = \tilde t_{11}, \ t_{-1}= \tilde t_{33}.
\end{align*}
Moreover the primary potential $F_0^{(\tilde{E}_7,G_{max})}(\tilde \bt)$ is holomorphic in 
\[
  {\CC^9 \times \left\lbrace \tilde t_{33} \in \CC \ | \ |\tilde t_{33}| < |\pi \Theta^2| \right\rbrace}
\] 
and has an expansion with rational coefficients.
\end{theorem}
Together with the explicit formulae for the genus $0$ small phase space potentials of the GW theories of the elliptic orbifolds announced in \cite{B1} this theorem gives an explicit closed formula for the FJRW potential of $(\tilde E_7, G_{max})$.
For example the following expansion holds:

\begin{align*}
F_0^{\tilde E_7,G_{max}} & = \frac{1}{2} \tilde t_{11}^2 \tilde t_{33}+\tilde t_{11} \left(\frac{\tilde t_{22}^2}{2}+\tilde t_{21} \tilde t_{23}+\tilde t_{13} \tilde t_{31}+\tilde t_{12} \tilde t_{32}\right)
  -\tilde t_{12}^2 \tilde t_{13} \left(\frac{\tilde t_{33}}{8}+\frac{\tilde t_{33}^5}{61440}\right)
  \\
  +\tilde t_{21}^2 \tilde t_{31} & \left(-\frac{\tilde t_{33}}{8}-\frac{\tilde t_{33}^5}{61440}\right)+\tilde t_{13} \tilde t_{21}^2 \left(\frac{1}{2}+\frac{\tilde t_{33}^4}{3072}+\frac{\tilde t_{33}^8}{330301440}\right)  \\
  +\tilde t_{12}^2 \tilde t_{31} & \left(\frac{1}{2}+\frac{\tilde t_{33}^4}{3072}+\frac{\tilde t_{33}^8}{330301440}\right)  \\
  +\tilde t_{12} &  \tilde t_{21} \tilde t_{22}\left(1+\frac{\tilde t_{33}^2}{32}+\frac{\tilde t_{33}^4}{6144}+\frac{\tilde t_{33}^6}{327680}+\frac{289 \tilde t_{33}^8}{2642411520}\right) + O(\tilde t_{33}^9,\tilde \bt_+^4)
\end{align*}
for $\tilde \bt_+ = \bt \backslash \tilde t_{33}$.

\subsection{Group action in the formal variable}
  
We make a few preparations, before we prove Theorem~\ref{theorem: CYLG}. The following Proposition appeared first in \cite{SZ1} in a slightly different notation.
\begin{proposition}[Section~3.2.3 of~\cite{SZ1}]\label{prop: wdvv of SZ}
Consider $x(q)$, $y(q)$ and $w(q)$ as the functions of $t=t_{-1}$ by taking $q = \exp(t_{-1})$.
The WDVV equation on $F_0^{\PP^1_{4,4,2}}$ is equivalent to the following system of equations:
\begin{equation}\label{eq: analytic group action}
  \begin{aligned}
	    \frac{\partial}{\partial t} x(t) &= x(t) \left( 2 y(t)^2  - x(t)^2 +w(t) \right), \\
	    \frac{\partial}{\partial t} y(t) &= y(t) \left( 2 x(t)^2 - y(t)^2 +w(t) \right), \\
	    \frac{\partial}{\partial t} w(t) &= w(t)^2-x(t)^4.
  \end{aligned}
  \end{equation}
\end{proposition}
  
The following proposition explains the $\SL(2,\CC)$-action we consider.
\begin{proposition}\label{prop: group action in formal variable}
Consider $x(q)$, $y(q)$ and $w(q)$ as the functions of $t=t_{-1}$ by taking $q = \exp(t_{-1})$. We have:
\begin{itemize}
\item[(i)] for any 
$A = \begin{pmatrix} 
    a & b \\ c & d
  \end{pmatrix} \in \mathrm{SL}(2, \mathbb{C})$ the functions $x^A(t)$, $y^A(t)$ and $w^A(t)$ defined by:
\begin{align*}
  x^A(t) &:= \frac{1}{(ct + d)} x \left(\frac{at + b}{ct + d}\right), \\
  y^A(t) &:= \frac{1}{(ct + d)} y \left(\frac{at + b}{ct + d}\right), \\
  w^A(t) &:= \frac{1}{(ct + d)^2} w \left(\frac{at + b}{ct + d}\right) - \frac{c}{ct + d},
\end{align*}
give solution to \eqref{eq: analytic group action}.

\item[(ii)] The potential $A \cdot F_0^{\PP^1_{4,4,2}}$ is obtained from $F_0^{\PP^1_{4,4,2}}$ by substituting:
\[
\{x(t_{-1}),y(t_{-1}),z(t_{-1})\} \to \{x^A(t_{-1}),y^A(t_{-1}),z^A(t_{-1})\}.
\]
\end{itemize}
\end{proposition}

\begin{proof}
  Part (i) is easy by using Proposition~\ref{prop: wdvv of SZ}, and part (ii) follows immediately from the definition of the $\SL(2,\CC)$-action on the primary potential, explicit form of $F_0^{\PP^1_{4,4,2}}$ and Proposition~\ref{prop: gw potential basics}.
\end{proof}

The following proposition will be used later.

\begin{proposition}\label{prop: Sl--action with the scalings}
For any $\alpha \in \CC^*$ and $A = \begin{pmatrix} 
    a & b \\ c & d
    \end{pmatrix} \in \mathrm{SL}(2, \mathbb{C})$ we have
\[
\left( \alpha x( \alpha^2 t) \right)^A = x^{A'}(t), \ \left( \alpha y( \alpha^2 t) \right)^A = y^{A'}(t), \ \left( \alpha^2 w( \alpha^2 t) \right)^A = w^{A'}(t),
\]
where 
$A' = \begin{pmatrix} 
    a\cdot\alpha & b\cdot\alpha \\ c/\alpha & d/\alpha
    \end{pmatrix} \in \mathrm{SL}(2, \mathbb{C})$.
\end{proposition}

\begin{proof}
First of all note that if $x(t)$, $y(t)$ and $w(t)$ give a solution to \eqref{eq: analytic group action}, then $\hat x(t) := \alpha x (\alpha^2 t)$, $\hat y(t) := \alpha y (\alpha^2 t)$, $\hat w(t) := \alpha^2 w (\alpha^2 t)$ is also a solution to \eqref{eq: analytic group action} so we can consider the action of $A$ from Proposition~\ref{prop: group action in formal variable}.

Indeed, 
\begin{align*}
\left( \alpha x(\alpha^2 t) \right)^A & = \frac{\alpha}{(ct + d)} x \left(\alpha^2 \cdot \frac{a t + b}{c t + d} \right) \\
      & = \frac{1}{(\frac{c}{\alpha} t + \frac{d}{\alpha})} x \left(\frac{\alpha a t + \alpha b}{\frac{c}{\alpha} t + \frac{d}{\alpha}} \right) = x^{A'}(t).
\end{align*}
    
The same computations are easy to perform for the remaining functions.
\end{proof}

In what follows we are going to consider the explicit values of the functions $\vartheta_k$ and make use of their holomorphicity. For such purposes it's convenient to write them not as the $q$-expansions, but as the holomorphic functions on $\HH$. The formal variable $t_{-1}$ is not suitable for these purposes. So we consider the changes of the variables $q = \exp \left(\frac{2 \pi \im\tau}{4} \right)$. This is equivalent to applying the change of variables $t_{-1} = 2 \pi \im\tau /4$ mentioned earlier. Applying it to the potential $F_0^{\PP^1_{4,4,2}}$ will change the terms defining the pairing. Because of this we give a special treatment to this change of the variables.

\subsection{Group action via the modular forms}
For $p \in \{2,3,4\}$ consider the following functions, holomorphic on $\HH$:
\[
\vartheta_p(\tau) := \vartheta_p(q(\tau)), \quad X_p^\infty(\tau) := 2 \frac{\partial}{\partial \tau} \log \vartheta_p(\tau).
\]

Fixing some branch of the square root, denote $\kappa:= \sqrt{2\pi \im / 4}$. We now introduce the new functions 
\begin{align*}
x^\infty(\tau) := \kappa \cdot x(q(\tau)), \quad y^\infty(\tau) := \kappa \cdot y(q(\tau)),   \\
    z^\infty(\tau) := \kappa \cdot z(q(\tau)), \quad w^\infty(\tau) := \kappa^2 \cdot w(q(\tau)), 
\end{align*}
  
For any $\tau_0 \in \HH$ and $\omega_0 \in \CC^*$ consider the functions $x^{(\tau_0,\omega_0)}$, $y^{(\tau_0,\omega_0)}$ and $z^{(\tau_0,\omega_0)}$:
\[
x^{(\tau_0,\omega_0)}(\tau) := \frac{2 \im\omega_0 {\rm Im}(\tau_0)}{ \left(2 \im\omega_0^2 {\rm Im}(\tau_0) - \tau\right)} 
    x^\infty \left( \frac{2 \im \omega_0^2 \tau_0 {\rm Im}(\tau_0) -\bar{\tau}_0 \tau }{2 \im\omega_0^2 {\rm Im}(\tau_0) -\tau} \right),
\]
with $y^{(\tau_0,\omega_0)}$, $z^{(\tau_0,\omega_0)}$ defined similarly, and also  
\begin{align*}
w^{(\tau_0,\omega_0)}(\tau) &:= \frac{\left( 2 \im\omega_0 {\rm Im}(\tau_0) \right)^2}{ \left(2 \im\omega_0^2 {\rm Im}(\tau_0) - \tau\right)^2} 
    w^\infty \left( \frac{2 \im \omega_0^2 \tau_0 {\rm Im}(\tau_0) -\bar{\tau}_0 \tau }{2 \im\omega_0^2 {\rm Im}(\tau_0) -\tau} \right) \\
    & - \frac{1}{ \left(2 \im\omega_0^2 {\rm Im}(\tau_0) - \tau\right)}.
\end{align*}

\begin{remark}
The functions introduced make sense from the point of view of modular forms; they are just expansions of the (quasi)modular forms $x(\tau)$, $y(\tau)$ and $w(\tau)$ at the point $\tau = \tau_0$ (see Proposition~17 in \cite{Z}). This is also a coordinate form of the Cayley transform of \cite{SZ2}.
\end{remark}

\begin{proposition}\label{prop: action in modular forms}
Fix some $\tau_0 \in \HH$ and $\omega_0 \in \CC^*$. We have:
\begin{itemize}
\item[(i)] The functions $x^{(\tau_0,\omega_0)}$, $y^{(\tau_0,\omega_0)}$, $w^{(\tau_0,\omega_0)}$ give a solutions to \eqref{eq: analytic group action}.

\item[(ii)] The functions $x^{(\tau_0,\omega_0)}(\tau)$, $y^{(\tau_0,\omega_0)}(\tau)$, $z^{(\tau_0,\omega_0)}(\tau)$, $w^{(\tau_0,\omega_0)}(\tau)$ are holomorphic on:
\[
D^{(\tau_0,\omega_0)} := \{\tau \in \CC \ | \ |\tau| < |2\omega_0^2 {\rm Im}(\tau_0)|\}.
\]
\item[(iii)] Consider the $\SL$-action on $x(t_{-1})$ as in Proposition~\ref{prop: group action in formal variable}. We have:
\[
\left( x(\tau) \right)^{(\tau_0,\omega_0)} = \left( x(t_{-1}) \right)^A,
\]
where
\[
A = 
	\begin{pmatrix}
	    \dfrac{\im\kappa \bar{\tau}_0}{2\omega_0{\rm Im}(\tau_0)} & \kappa \omega_0 \tau_0\\
	    \dfrac{\im}{2 \kappa \omega_0{\rm Im}(\tau_0)} & \dfrac{\omega_0}{\kappa }
	\end{pmatrix}.
\]
\end{itemize}
\end{proposition}

\begin{proof}
Part (i) is easily checked by the explicit differentiation and definition of the functions $x^{(\tau_0,\omega_0)}(\tau)$, $y^{(\tau_0,\omega_0)}(\tau)$, $w^{(\tau_0,\omega_0)}(\tau)$.  Part (ii) follows from the fact that the theta constants are holomorphic functions on $\HH$.

For part (iii) note first that in principle the action of Proposition~\ref{prop: group action in formal variable} is more general. It can be applied to any solution of \eqref{eq: analytic group action}. The rest follows from Proposition~\ref{prop: Sl--action with the scalings}.
\end{proof}

\begin{remark}
  The action $x^\infty \to x^{(\tau_0,\omega_0)}$ can be seen as the action changing the primitive form of the B model. Having applied this action on the B side we get the CohFT of a simple elliptic singularity $\tilde E_7$ with the primitive form ``at $\tau_0$'' (see \cite{BT,MR,B1}).
\end{remark}

\subsection{Proof of Theorem~\ref{theorem: CYLG}}
First of all note that the change of variables of Theorem~\ref{theorem: CYLG} identifies the two pairings. This is clear also that the action of any $A \in \SL(2,\CC)$ doesn't change the correlators involving insertion of the unit vector of a CohFT.

Applying the linear change of the variables $\tilde \bt = \tilde \bt(\bt)$ given in the theorem to $A \cdot F_0^{\PP^1_{4,4,2}}$ we get:
\begin{align*}
    &A \cdot F_0^{\PP^1_{4,4,2}}(\tilde \bt) = \frac{1}{2} \tilde t_{11}^2 \tilde t_{33} + \tilde t_{11} \left(\tilde t_{21} \tilde t_{23} + \tilde t_{13} \tilde t_{31} + \tilde t_{12} \tilde t_{32} + \frac{\tilde t_{22}^2}{2}\right) 
    \\
    \quad &+ \frac{1}{2}  \tilde t_{21}^2 \left( C_1(\tilde t_{33}) \cdot \tilde t_{13} + C_2(\tilde t_{33}) \cdot \tilde t_{31}\right)
    + \frac{1}{2} \tilde t_{12}^2 \left( C_2(\tilde t_{33}) \cdot \tilde t_{13} + C_1(\tilde t_{33}) \cdot \tilde t_{31}\right) 
    \\
    \quad & + \sqrt{2} \left(x^A(\tilde t_{33}) + y^A(\tilde t_{33})\right) \tilde t_{12} \tilde t_{21} \tilde t_{22}  + O(\tilde t_{i,k}^4, \tilde t_{11})
\end{align*}
for $C_1(\tilde t_{33}) := x^A(\tilde t_{33}) -y^A(\tilde t_{33}) +z^A(\tilde t_{33})$ and $C_2(\tilde t_{33}) := x^A(\tilde t_{33}) -y^A(\tilde t_{33}) -z^A(\tilde t_{33})$.   
  
\begin{lemma}
The equality of the formal series     
$A \cdot F_0^{\PP^1_{4,4,2}}(\tilde \bt) = F_0^{\tilde E_7, G_{max}}(\tilde \bt)$
is satisfied if any only if 
\[
A \cdot F_0^{\PP^1_{4,4,2}}(\tilde \bt) - F_0^{\tilde E_7, G_{max}}(\tilde \bt) \in O(\tilde t_{i,k}^4, \tilde t_{11}).
\]
\end{lemma}

\begin{proof}
One direction is straightforward and we concentrate on the opposite one.
    
First of all note that the potential $A \cdot F_0^{\PP^1_{4,4,2}}(\tilde \bt)$  satisfies the same quasihomogeneity property as the potential $F_0^{\PP^1_{4,4,2}}(\tilde \bt)$. Next one sees easily that the change of the variables $\tilde \bt(\bt)$ preserves the quasihomogeneity property. 

Recall the genus zero reconstruction lemma of Subsection~\ref{subsection: reconstruction in genus zero}.
The equality above assures also that the algebra structure at the origin coincides on the both sides. Hence the notion of the primitive vectors coincides on the both sides.

Hence the conditions of Lemma~\ref{lemma: reconstruction in genus zero} coincide for the both potentials. For the FJRW theory these conditions were written in Lemma~\ref{lemma:reconstruction of FJR} to be exactly those as described in the proposition.
\end{proof}

However in order to prove the theorem we need to use explicit values of the functions and therefore work with the ``modular'' variable $\tau \in \HH$.
  
\begin{lemma}\label{l:desired expansions}
Let $\tau_0 = \im$ and $\omega_0 := \kappa \sqrt{2\pi} / \left( \Gamma (\frac{3}{4}) \right)^2$. The equation 
\[
A^{CY/LG}\cdot~F_0^{\PP^1_{4,4,2}}(\tilde \bt)=F_0^{\tilde E_7, G_{max}}(\tilde \bt)
\]
hold if and only if:

\begin{equation}\label{eq: desired expansions}
\begin{aligned}
	x^{(\tau_0,\omega_0)}(\tau) - y^{(\tau_0,\omega_0)}(\tau) + z^{(\tau_0,\omega_0)}(\tau) &= 1 + O(\tau^2),
	\\
	x^{(\tau_0,\omega_0)}(\tau) - y^{(\tau_0,\omega_0)}(\tau) - z^{(\tau_0,\omega_0)}(\tau) &= - \frac{\tau}{4} + O(\tau^2),
	\\
	x^{(\tau_0,\omega_0)}(\tau) + y^{(\tau_0,\omega_0)}(\tau) &= \frac{1}{\sqrt{2}} + O(\tau^2).
\end{aligned}
\end{equation}
\end{lemma}

\begin{proof}
By using the lemma above and reconstruction Lemma~\ref{lemma: reconstruction in genus zero} we see that it's enough to compare the potentials $F_0^{\tilde E_7,G_{max}}$ and $A \cdot F_0^{\PP^1_{4,4,2}}(\tilde \bt)$ up to $O(\tilde t_{i,k}^4, \tilde t_{11})$.
    
Recall part~(iii) of Proposition~\ref{prop: action in modular forms}. Note that for the $\tau_0$ and $\omega_0$ as in the statement of the Lemma, the matrix $A'$ coincides with the matrix $A^{CY/LG}$.
    
The equalities above are obtained by comparing the coefficients of $F_0^{\tilde E_7, G_{max}}$ and $A\cdot F_0^{\PP^1_{4,4,2}}(\tilde \bt)$. The RHS of them are taken from the explicit form of $F_0^{\tilde E_7, G_{max}}$ (recall Section~\ref{subsection: E_7 state space}).
  
It follows from Lemma~\ref{lemma:reconstruction of FJR} that it's enough to check these equalities in order for the whole potentials to coincide.
\end{proof}

In the remainder of this section we show that \eqref{eq: desired expansions} is satisfied by the functions $x^{(\tau_0,\omega_0)}(\tau)$, $y^{(\tau_0,\omega_0)}(\tau)$, $z^{(\tau_0,\omega_0)}(\tau)$ for $\tau_0$ and $\omega_0$ as in Lemma~\ref{l:desired expansions}.
  
Denote by $\tilde x, \tilde y, \tilde z$ the expansion of the function $x$, $y$, $z$ with the change of variables $\tau \to A^{(\tau_0,\omega_0)} \tau$ applied, i.e. 
\[
\tilde x (\tau) := x^\infty \left( \frac{2\im\omega_0^2 \tau_0 {\rm Im}(\tau_0) -\bar{\tau}_0 \tau }{2\im\omega_0^2 {\rm Im}(\tau_0) - \tau} \right),
\]
and similar for $\tilde y$, $\tilde z$. Define the numbers $x_0,x_1$,$y_0,y_1$ and $z_0,z_1$ as the coefficients of the series expansions at $\tau = 0$:

\begin{align*}
\tilde x = x_0 + x_1 \tau + O(\tau^2), \quad \tilde y = y_0 + y_1 \tau + O(\tau^2), \quad \tilde z = z_0 + z_1 \tau + O(\tau^2).
\end{align*}

The functions $x^{(\tau_0,\omega_0)}$, $y^{(\tau_0,\omega_0)}$, $z^{(\tau_0,\omega_0)}$ satisfy
\[
x^{(\tau_0,\omega_0)}(\tau) = \frac{x_0}{\omega_0} + \tau \left( \frac{x_1}{\omega_0} + \frac{x_0}{2 \im \omega_0^3 {\rm Im}(\tau_0)} \right) + O(\tau^2).
\]

To find the coefficients explicitly we use the following derivation formula.

\begin{lemma}
The derivatives of the functions $x$, $y$, $z$ satisfy:
\begin{align*}
      \frac{\partial}{\partial \tau} \tilde x(\tau)\mid_{\tau = 0} & = \frac{\kappa}{2\omega_0^2} \left( \vartheta_3(\tau_0)^2 X_3^\infty(\tau_0) + \vartheta_4(\tau_0)^2 X_4^\infty(\tau_0) \right),
      \\
      \frac{\partial}{\partial \tau} \tilde y(\tau)\mid_{\tau = 0} &= \frac{\kappa}{2\omega_0^2} \left( \vartheta_3(\tau_0)^2 X_3^\infty(\tau_0) - \vartheta_4(\tau_0)^2 X_4^\infty(\tau_0) \right),
      \\
      \frac{\partial}{\partial \tau} \tilde z(\tau)\mid_{\tau = 0} &= \frac{\kappa}{\omega_0^2}  \vartheta_2(\tau_0)^2 X_2^\infty(\tau_0).
\end{align*}
\end{lemma}

\begin{proof}
By using the double argument formulae of the Jacobi theta constants we see:
\begin{align*}
      2 x \left( q \right) &= \vartheta_3(q^4)^2 + \vartheta_4(q^4)^2, 
      \\
      2 y \left( q \right) &= \vartheta_3(q^4)^2 - \vartheta_4(q^4)^2,
\end{align*}
and all functions $x(q)$, $y(q)$, $z(q)$ are written via $q^4$. Directly from the definition of $X_k^\infty(\tau)$ and the rescaling we get:

\begin{align*}
      \frac{\partial}{\partial \tau} x^\infty(\tau) & = \frac{\kappa}{2}\left( \vartheta_3(\tau)^2 X_3^\infty(\tau) + \vartheta_4(\tau)^2 X_4^\infty(\tau) \right),
      \\
      \frac{\partial}{\partial \tau} y^\infty(\tau) &= \frac{\kappa}{2}\left( \vartheta_3(\tau)^2 X_3^\infty(\tau) - \vartheta_4(\tau)^2 X_4^\infty(\tau) \right),
      \\
      \frac{\partial}{\partial \tau} z^\infty(\tau) &=  \kappa\vartheta_2(\tau)^2 X_2^\infty(\tau).
\end{align*}

The rest follows from the chain rule and the definition of $\tilde x$, $\tilde y$, $\tilde z$.
\end{proof}

The values of the theta constants and their logrithimic derivatives at the point $\tau = \im$ are known to be:
\begin{align*}
    \vartheta_2 \left(\im\right) = \frac{\pi^{1/4}}{2^{1/4} \Gamma \left( \frac{3}{4} \right)}, \quad \vartheta_3 \left(\im\right) &= \frac{\pi^{1/4}}{\Gamma \left( \frac{3}{4} \right)}, \quad \vartheta_4 \left(\im\right) = \frac{\pi^{1/4}}{2^{1/4} \Gamma \left( \frac{3}{4} \right)},
    \\
    X_2^\infty(\im) = \frac{\im \pi^2}{4 \left( \Gamma \left(\frac{3}{4} \right) \right)^4} + \frac{\im}{2}, \quad X_3^\infty(\im) &= \frac{\im}{2},\quad  X_4^\infty(\im) = - \frac{\im \pi^2}{4 \left( \Gamma\left(\frac{3}{4} \right) \right)^4} + \frac{\im}{2}.
\end{align*}

For $K = \frac{\pi^{1/2}}{2^{1/2} \left(\Gamma \left( \frac{3}{4} \right)\right)^2}$ using the lemma above we get: 
\begin{align*}
    & \kappa^{-1}\tilde x(\tau) = \frac{K}{2} \left(\sqrt{2} + 1\right) + \tau \frac{K \im}{2 \omega_0^2} \left( \frac{\sqrt{2}}{2} + \left( \frac{1}{2} - \frac{K^2\pi}{2} \right)\right)
    + O(\tau^2),
    \\
    & \kappa^{-1}\tilde y(\tau) = \frac{K}{2} \left(\sqrt{2} - 1\right) + \tau \frac{K \im}{2 \omega_0^2} \left( \frac{\sqrt{2}}{2} - \left( \frac{1}{2} - \frac{K^2\pi}{2} \right)\right)
    + O(\tau^2),
    \\
    & \kappa^{-1}\tilde z(\tau) = K  + \tau \frac{K \im}{\omega_0^2}  \left(\frac{1}{2} + \frac{\pi  K^2}{2}\right) + O(\tau^2).
\end{align*}

Hence
\begin{align*}
      \kappa^{-1} &\left( x^{(\im,\omega_0)}(\tau) - y^{(\im,\omega_0)}(\tau) + z^{(\im,\omega_0)}(\tau) \right) 
      \\
      &= \frac{x_0 - y_0 + z_0}{\omega_0} + \tau \left( \frac{x_1 - y_1 + z_1}{\omega_0} + \frac{x_0 - y_0 + z_0}{2 \im \omega_0^3} \right)
      + O(\tau^2)
      \\
      &= 2\frac{K}{\omega_0} + O(\tau^2),
      \\
      \kappa^{-1} &\left( x^{(\im,\omega_0)}(\tau) - y^{(\im,\omega_0)}(\tau) - z^{(\im,\omega_0)}(\tau) \right) 
      \\
      &= \frac{x_0 - y_0 - z_0}{\omega_0} + \tau \left( \frac{x_1 - y_1 - z_1}{\omega_0} + \frac{x_0 - y_0 - z_0}{2 \im \omega_0^3} \right)
      + O(\tau^2)
      \\
      &= - \frac{\tau \cdot \pi \im K^3}{\omega_0^3}       + O(\tau^2),
      \\
      \kappa^{-1} &\left( x^{(\im,\omega_0)}(\tau) + y^{(\im,\omega_0)}(\tau) \right) 
      \\
      &= \frac{x_0 + y_0}{\omega_0} + \tau \left( \frac{x_1 + y_1}{\omega_0} + \frac{x_0 + y_0}{2 \im \omega_0^3} \right)
      + O(\tau^2)
      \\
      &= \sqrt{2} \frac{K}{\omega_0} + O(\tau^2).
\end{align*}

Fixing $\omega_0 = 2 K \kappa$ we get exactly the expansions as in \eqref{eq: desired expansions}. This completes proof of Theorem~\ref{theorem: CYLG}.

\section{Givental's action and CY/LG correspondence}\label{section: GiventalsAction}

In this section we formulate the CY/LG correspondance via the group action on the space of cohomological field theories and give the particular action, connecting $\cF_0^{\PP^1_{4,4,2}}$ and $\cF_0^{(\tilde E_7,G_{max})}$.

\subsection{Inifinitesimal version of Givental's action}
  
In this subsection we introduce Givental's group action on the partition function of a CohFT via the inifinitesimal action computed in \cite{L}. Let $\Lambda_{g,n}$ be a unital CohFT on $(V, \eta)$ with the unit $e_0 \in V$.
    
The \textit{upper-triangular group} consists of all elements $R = \exp(\sum_{l=1} r_l z^l)$, such that
\[
r(z) = \sum_{l \ge 1} r_l z^l \in {\rm Hom}(V,V) \otimes \mathbb C[z],
\] 
and $r(z) + r(-z)^* = 0$ (where the star means dual with respect to $\eta$). Following Givental, we define the \textit{quantization} of $R$:
\[
\hat R := \exp( \sum_{l=1} \widehat{r_l z^l}),
\]
where for $(r_l)^{\alpha,\beta} = (r_l)^\alpha_\sigma \eta^{\sigma, \beta}$ we have:
\begin{equation*}
\begin{aligned}
	\widehat {r_lz^l} := & -(r_l)_1^\alpha \frac{\partial}{\partial t^{l+1,\alpha}} + \sum_{d =0}^\infty t^{d, \beta} (r_l)_\beta^\alpha \frac{\partial}{\partial t^{d+l, \alpha}}
	  \\
	  & + \frac{\hbar}{2} \sum_{i+j=l-1} (-1)^{i+1} (r_l)^{\alpha,\beta} \frac{\partial^2}{\partial t^{i,\alpha} t^{j,\beta}},
\end{aligned}
\end{equation*}

The following theorem is essentially due to Givental.

\begin{theorem}[\cite{G}]
  The differential operator $\hat R$ acts on the space of partition functions of CohFTs.
\end{theorem}

The action of $\hat R$ can be also formulated on the CohFT itself---not just on its partition function (cf. \cite{PPZ}).
We call the action of the differential operator $\hat R$ on the partition function of the CohFT \emph{Givental's $R$-action} or upper-triangular Givental's group action.

The \textit{lower-triangular group} consists of all elements $S = \exp(\sum_{l=1} s_l z^{-l})$, such that 
\[
	s(z) = \sum_{l \ge 1} s_l z^{-l} \in {\rm Hom}(V,V) \otimes \mathbb C[z^{-1}]
\] 
and $s(z) + s(-z)^* = 0$. Following Givental, we define \textit{the quantization} of $S$:
\[
\hat S := \exp( \sum_{l=1}^\infty (s_lz^{-l})\hat{\ } ),
\]
where
\begin{align*}
	\sum_{l=1}^\infty &(s_lz^{-l})\hat{\ } =  -(s_1)_1^\alpha \frac{\partial}{\partial t^{0,\alpha}}
	+ \frac{1}{\hbar} \sum_{d=0}^{\infty} (s_{d+2})_{1,\alpha} \, t^{d,\alpha}
	\\ 
	& + \sum_{ \substack{d=0\\ l=1} }^\infty
	(s_l)_\beta^\alpha \, t^{d+l,\beta} \frac{\partial}{\partial t^{d,\alpha}}
	+ \frac{1}{2 \hbar} \sum_{ \substack{d_1,d_2 \\ \alpha_1,\alpha_2} }
	(-1)^{d_1} (s_{d_1+d_2+1})_{\alpha_1,\alpha_2} \, t^{d_1,\alpha_1} t^{d_2,\alpha_2}.\notag
\end{align*} 

In contrast to $\hat R$, the action of the differential operator $\hat S$ generally\footnote{See for example \cite[Section~1]{PPZ} } can't be extended to the action on the space of CohFTs. Moreover, it could happen that $\hat S \cdot \cZ^\Lambda$ is not anymore a partition function in our definition\footnote{One can consider $\hat S$ as acting on the space of genus zero potentials, if one treats the latter one as a space of functions, subject to Dilaton, String and TRR--0 equation.}. However for the examples of this paper $\hat S\cdot \cZ^\Lambda$ is still a partition function. In general such $\hat S$ are mostly used to perform linear change of the variables, however they can also affect $1$--point and $2$--point correlators.

We call $\hat S: \cZ^\Lambda \to \hat S \cdot \cZ^\Lambda$ the \emph{lower-triangular} Givental's group action. 

\subsection{$R$-matrix of a CohFT}\label{ss:gw theory of point}
Fix a unital CohFT $\Lambda$ on $(V,\eta)$ with unit $e_0$ and $m+1=\dim V$. Let $\bt = (t^0,\dots,t^m)$ with $t^\alpha := t^{0,\alpha}$ as in Section~\ref{ss:cohft axioms}. 
For each choice of indices $i,j,k \in 0,\dots,m$ define:
\[
  c_{ij}^k(\bt) := \sum_{p=0}^m\frac{\partial^3 F_0}{\partial t^i \partial t^j \partial t^p} \eta^{pk}.
\]
Because $F_0$ is a solution to WDVV equation (see Section~\ref{subsection: reconstruction in genus zero}) the functions $c_{ij}^k(\bt)$ are structure constants of an associative and commutative algebra for all $\bt$. Denoting the basis of this algebra by $\langle\partial/\partial t^0,\dots,\partial /\partial t^m\rangle$, the product $\circ$ reads:
\[
  \frac{\partial}{\partial t^i} \circ  \frac{\partial}{\partial t^j} = \sum_{k=0}^m c_{ij}^k \frac{\partial}{\partial t^k}.
\]
Moreover this algebra turns out to be a Frobenius algebra with respect to the pairing $\eta$.

The CohFT $\Lambda$ is called \textit{semisimple} if the algebra defined by $c_{ij}^k(0)$ above is semisimple. In that case there are new coordinates $u^0(\bt),\dots,u^m(\bt)$, such that $\partial/\partial_{u_i} \circ \partial/\partial_{u_j} = \delta_{i,j}\Delta_i^{-1} \partial/\partial_{u_i}$ for some functions $\Delta_i = \Delta_i(\bu)$. Let $\Psi$ be the transformation matrix from the frame $\langle \partial/\partial t^0, \dots, \partial/\partial t^m \rangle$ to the frame $\langle \partial/\partial u^0, \dots, \partial/\partial u^m \rangle$.

Consider the partition function $\cZ^{pt}$ of the GW theory of a point.
This is a partition function of a CohFT on a one-dimensional space, and can be therefore written in coordinates $\{u^{\ell,0}\}_{\ell \ge 0}$.
In the next formula take the product of $m+1$ such partition functions indexing however the variables.
\begin{equation}\label{eq: tauOfKdv}
  \mathcal{T}^{(m+1)} = \prod_{k=0}^m \cZ^{pt}\left(\{u^{\ell,k}\}_{\ell \ge 0} \right).
\end{equation}
Consider also
\[
  \hat \Delta \cdot \mathcal{T}^{(m+1)} := \prod_{k=0}^m \cZ^{pt}\left(\{u^{\ell,k}\}_{\ell \ge 0} \right) \mid_{u^{\ell,k} \to  \Delta_i^{1/2} v^{\ell,k}, \ \hbar \to \Delta_i \hbar  }
\]

The following theorem was conjectured by Givental and later proved by Teleman.

\begin{theorem}[Theorem~1 in \cite{T}]\label{theorem: teleman}
For every quasihomogenous semisimple unital CohFT $\Lambda$ on an $m+1$-dimensional vector space $V$ there is unique upper-triangular group element $R$, such that
\[
\cZ^\Lambda = \hat R \cdot \hat \Psi \cdot \hat \Delta \cdot \mathcal{T}^{(m+1)},
\]
where $\hat \Psi$ acts by the change of the variables $v^{d,\alpha} = \Psi_\beta^\alpha t^{d,\beta}$.
\end{theorem}

It often happens that although the structure constants $c_{ij}^k(0)$ do not define a semisimple algebra, there is $\bt_0$, s.t. for $\bt = \bt_0$ the semisimplicity condition holds. In this case one can consider the $S$--action, acting on $\cZ^\Lambda$ by just the change of the variables $\bt \to \bt - \bt_0$, allowing one to apply Theorem~\ref{theorem: teleman} to $\hat S\cdot \cZ^\Lambda$. 

Because $\hat \Psi$ and $\hat \Delta$ only apply the change of the variables, the most important part of the formula above is located in the action or $R$. This motivates the following definition.

\begin{definition}
The upper-triangular group element $R$ as above is called the \textit{R-matrix} of the CohFT $\Lambda$.
\end{definition}

In order to find such an $R$-matrix explicitly one would normally use the recursive procedure described by Givental. After writing $R = \mathrm{Id} + \sum_{k \ge 1} R_kz^k$ every matrix $R_k$ is uniquely determined by the preceding matrices. However, it is difficult to perform this procedure to the end to have a closed formula for $R = R(z)$. Up to now the only explicitly written $R$-matrix is for the theory of 3-spin curves, which is two-dimensional (cf. \cite{PPZ}).

Furthermore, it could still happen that the formula of the theorem above holds for a non-semisimple CohFT. In this case one doesn't know if the $R$-matrix is unique and the recursive procedure above can no longer be applied.
We will return to this question in Section~\ref{section: the last one}, where we present a formula similar to the $R$-matrix for the Gromov--Witten theory of $\PP^1_{4,4,2}$ without the use of the recursive procedure described above.

\subsection{Mirror symmetry and CY/LG correspondence}
The CY/LG correspondence is best understood using mirror symmetry via the B model. 

Given a hypersurface singularity $\tilde W: \CC^N \to \CC$ one can construct the so-called Saito--Givental CohFT, which depends non-trivially on the certain special choice of $\zeta$ a \emph{primitive form} of Saito. Let $\cZ_{\tilde W,\zeta}$ be the partition function of this CohFT. 

CY--LG mirror symmetry conjectures that the partition function $\cZ_{\tilde W,\zeta_\infty}$ with the special choice of the primitive form $\zeta = \zeta_\infty$ coincides with the partition function of the GW theory of some Calabi--Yau variety $\cX$ up to probably a linear change of the variables.

LG--LG mirror symmetry conjectures that the partition function $\cZ_{\tilde W,\zeta_0}$ with the another special choice of the primitive form $\zeta = \zeta_0$ coincides with the partition function of the FJRW theory of some pair $(W,G_{max})$ up to probably a linear change of the variables, where $W: \CC^N \to \CC$ is some other hypersurface singularity (generally different from $\tilde W$).

One says than that the GW theory of $\cX$ and FRJW theory of $(W,G_{max})$ constitute two mirror A models 
of the one B model of $\tilde W$, taken in the different phases---$\zeta_\infty$ and $\zeta_0$. This lead to the following conjecture.

\begin{conjecture}\label{conj:Amodel Rmatrix}
Let GW theory of $\cX$ and FJRW theory of $(W,G_{max})$ be two mirror A models of the same B model. Then there is an upper-triangular Givental's action $R=R(z)$, such that
\[
\hat R \cdot \cZ^{\cX} (\bt) = \cZ^{(W,G_{max})} (\tilde\bt(\bt)),
\]
where $\tilde \bt = \tilde \bt (\bt)$ is a linear change of the variables.
\end{conjecture}

When two mirror symmetry conjectures of type CY--LG and type LG--LG hold, Conjecture~\ref{conj:Amodel Rmatrix} is an A side analogue of the following B side conjecture: 

\begin{conjecture}there is an upper-triangular group element of Givental $R=R(z)$, such that up to a linear change of variables the following equation holds:
\[
\hat R \cdot \cZ_{\tilde W,\zeta_\infty} = \cZ_{\tilde W,\zeta_0}.
\] 
\end{conjecture}

In the case of simple elliptic singularities this sort of action was investigated in \cite{MR,BT,B1}. In particular, it was shown in \cite{B1} that the $\SL(2,\CC)$-action of Section~\ref{section: CYLG via modularity} has at the same time the meaning of the primitive form change on the B side and can be written via the certain $R$-action of Givental. In other words, \eqref{e:A Action} can be realized as the restriction of the certain action of Givental to the small phase space. 

\subsection{CY/LG correspondence via Givental's action}
For any $\tau, \sigma\in \CC$ consider the lower-triangular group element 
\[
    S^\tau(z) = 
    \exp \left( \Bigg(
    \begin{array}{c c c}
      0 & \dots & 0 \\
      \vdots & 0 & \vdots \\
      \tau & \dots & 0
    \end{array}\Bigg)
    z^{-1}\right),
\]
and the upper-triangular group element $R^\sigma$:
\[
    R^\sigma(z) = 
    \exp \left( \Bigg(
    \begin{array}{c c c}
      0 & \dots & \sigma \\
      \vdots & 0 & \vdots \\
      0 & \dots & 0
    \end{array} \Bigg)
    z\right) .
\]

For any $c\in \CC$, we also define the matrix 
\[
S_0^c := 
   \left(
     \begin{array}{c c c}
       1 & \dots & 0 \\
       \vdots & c \cdot I_{n-2} & \vdots \\
       0 & \dots & c^2
     \end{array}
   \right)
\]
together with an action of $S_0^c$ on $\cZ(\bt)$ (which we denote by $\hat{S}_0^c$) defined by 
\[
t^{\ell,\alpha} \to (S_0^c)_\beta^\alpha \ t^{\ell,\beta} \quad \text{and} \quad \hbar \to c^2 \hbar.
\]

Let $\Theta = \frac{\sqrt{2\pi}}{\left(\Gamma \left( \frac{3}{4} \right)\right)^2}$ as in Theorem~\ref{theorem: CYLG}, define:
\[
  \tau := -\frac{\pi}{2}, \quad \sigma := -\frac{1}{\pi\Theta^2}, \quad c:= \frac{1}{\Theta}.
\]

We give now the Givental's action form of the CY/LG correspondence in genus zero.

\begin{theorem}\label{theorem: cylg via Givental}
Consider the partition functions $\cZ^{(\tilde E_7,G_{max})}$ and $\cZ^{\PP^1_{4,4,2}}$. We have:
\[
\cF_0^{(\tilde{E}_7,G_{max})} = \res_\hbar \ln \left( \hat R^\sigma \cdot \hat S_0^c \cdot \hat S^\tau \cdot \cZ^{\PP^1_{4,4,2}} \right)
\]
with the Givental's element $S^{\tau}$, $S_0^c$ and $R^\sigma$ defined above.
\end{theorem}

\begin{proof}
One can check (cf. \cite{B1}) that the action of the theorem induces the action of Theorem~\ref{theorem: CYLG} on the primary genus 0 potentials. And in Theorem~\ref{theorem: CYLG} and Corollary~5.1 in \cite{B1} we see that the theorem holds for the primary potentials. 
We only have to take care of the psi-classes insertions.
  
However in genus zero all correlators are unambiguesly reconstructed from the small phase space correlators by using the topological recursion relation. Hence we can reconstruct these correlators on the LHS from the small phase space.
\end{proof}

\section{Extended FJRW correlators}\label{section: twisted theory}
In this section, we reformulate FJRW theory in order to obtain the genus zero potential $\cF_0^{(\tilde{E}_7,G_{max})}$ from a basic CohFT. This method is also used in \cite{ChR}, \cite{ChIR}, \cite{PrSh}, so we will be brief. The details can be found in these other articles. In this section and the section following, we fix $W=x^4+y^4+z^2$, and $G=G_{max}$.   
\subsection{$r$-spin theory}
Let $(A_r)_{g,n}$ denote the moduli space of genus $g$, $n$-marked $A_r$-curves corresponding to the polynomial $A_r=x^{r+1}$. Such $W$-structures are often referred to as $r$-spin curves.  Let $(A_W)_{g, n}$ denote the fiber product
\[
(A_W)_{g, n}:=
(A_3)_{g,n}\times_{\cM_{g,n,4}}(A_3)_{g,n}\times_{\cM_{g,n,4}}(A_1)_{g,n}
\]

\begin{proposition}[\cite{ChR}]There is a surjective map
\[
s:(A_W)_{g, n}\to \cW_{g,n}
\]
which is a bijection at the level of a point. 
\end{proposition}

Each factor of $(A_r)_{g,n}$ in the product above is equipped with a universal $A_r$-structure. Abusing notation, we denote the universal line bundle over the $k$th factor of $(A_W)_{g, n}$ also by $\LL_k$.
By the universal properties of the $W$-structure on $\cW_{g,n}$, we have $s^*\LL_k\cong \LL_k$ for $1 \leq k \leq 3$. 

Given $\bv{h}=(h_1,\dots,h_n)$, let us denote 
\begin{align*}
&\cA_W(\bv{h})_{g,n}:= \\
 &\quad(A_3)_{g,n}(\varTheta_1^{h_1}, \dots \varTheta_1^{h_n})\times_{\cM_{g,n,4}}\dots \times_{\cM_{g,n,4}}(A_1)_{g,n}(\varTheta_3^{h_1},\dots,\varTheta_3^{h_n}).
\end{align*}

By the projection formula, we can pull back to $(A_W)_{0,n}$, and obtain the following expression for the genus $0$ correlators:
\begin{align*}
&\br{\tau_{a_1}(\phi_{h_1}),\dots, \tau_{a_n}(\phi_{h_n})}^{(\tilde E_7,G_{max})}_{0,n}= 32\int_{A_W(\bv{h})_{0,n}} \prod_{i=1}^n \psi_i^{a_i}\cup c_{top}\Big(R^1\pi_*\big(\bigoplus_{i=1}^3 \LL_i\big)^\vee\Big)
\end{align*}
The factor in front of the integral is the factor in \eqref{e:fjrcohft}.

From this description, we see that the FJRW theory in this case is a so-called \emph{twisted} theory. Thus we can use Givental's formalism to give an expression for the generating function of these correlators.  

\subsection{Twisted theory}\label{section: twistedTheory}
We will construct a \emph{twisted} FJRW theory whose correlators coincide with those of $(\tilde E_7,G_{max})$ in genus zero.
We first extend the state space
\[
\sH_{W,G}^{ext}:=\sH_{W,G}\oplus\bigoplus_{h\in G \setminus G^{nar}} \CC\cdot \phi_h.
\]
Any point $\bt\in\sH_{W,G}^{ext}$ can be written as $\bt=\displaystyle \sum_{h \in G}t^h\phi_h$. 
Let $i_k(h) := \langle \varTheta_{k}^h- q_k \rangle$, where $\langle - \rangle$ denotes the fractional part. Notice $i_k(h)=1-q_k$ exactly when $\varTheta_k^h=0$. 
Set 
\[
\deg_W(\phi_h) := 2\sum_{k=1}^3 i_k(h).
\]
For $h \in G^{nar}$, this definition matches the W-degree defined in \eqref{e:degw}.

We extend the definition of our FJRW correlators to include insertions $\phi_h$ in $\sH_{W,G}^{ext}$.  Namely, set 
\[
\br{\tau_{a_1}(\phi_{h_1}),\dots, \tau_{a_n}(\phi_{h_n})}_{0,n}^{(W,G)} = 0
\] 
if $h_i \in G \setminus G^{nar}$ for some $i$.

We would like to unify our definition of the extended FJRW correlators, by re-expressing them as integrals over $ (\widetilde A_W)_{0,n}$, a slight variation of $ (A_W)_{0,n}$, where instead of considering orbifold line bundles, we consider line bundles on the coarse curve with multiplicities (see the discussion prior to Proposition~\ref{p:linebundledegree}).
We will make use of the following lemma.

\begin{lemma}[\cite{ChR}]\label{l:linebndles}
Let $\cC$ be a $d$-stable curve with coarse underlying curve $C$, and let $M$ be a line bundle pulled back from $C$. If $l|d$, there is an equivalence between two categories of $l$th roots $\cL$ on $d$-stable curves:
\[
\set{\cL|\cL^{\otimes l}\cong M}\leftrightarrow \bigsqcup_{0\leq E<  \sum lD_i}\set{\cL|\cL^{\otimes l}\cong M(-E), \mult_{p_i}(\cL)=0}.
\]
where the union is taken over divisors $E$ which are linear combinations of (integer) divisors $D_i$ corresponding to the marked points $p_i$. 
\end{lemma}

\begin{proof}
Let $p$ denote the map which forgets stabilizers along the markings.  The correspondence is simply $\cL \mapsto p^* p_*(\cL)$.  
\end{proof}

\begin{definition}For $m_1,\dots, m_n\in \{\frac 14, \frac 12, \frac 34, 1\}$, consider the stack $\tilde A_3\left(m_1,\dots,m_n\right)_{g,n}$ classifying genus $g$, $n$-pointed, 4-stable curves equipped with fourth roots:
\begin{align*}
&\widetilde A_3\left(m_1,\dots,m_n\right)_{g,n}:=\\
&\quad\set{(\cC, p_1,\dots, p_n, \cL, \varphi)|\phi:\cL^{\otimes 4}\stackrel{\sim}{\to} \omega_{\log}(-\sum_{i=1}^n4m_iD_i),\; \mult_{p_i}(\cL)=0}, 
\end{align*}
where the integer divisors $D_i$ correspond to the markings $p_i$. 
\end{definition}
The moduli space $\widetilde A_3\left(m_1,\dots,m_n\right)_{g,n}$ also has a universal curve $\sC\to \widetilde{A}_3$ and a universal line bundle $\widetilde{\LL}$. We can define everything similarly for $A_1$ and replace it with $\widetilde{A}_1$.  

We now define an analogue of $(A_W)_{g,n}$, replacing $(A_3)_{g,n}$ with $(\widetilde{A}_3)_{g,n}$ in the first two factors, and $(A_1)_{g,n}$ with $(\widetilde{A}_1)_{g,n}$.  
For $1 \leq i \leq n$, let $m_i = (m_{1i}, \ldots, m_{3i})$ be a tuple of fractions satisfying $m_{1i},m_{2i} \in \{\frac 14, \frac 12, \frac 34, 1\}$, and $m_{3i} \in \{\frac 12, 1\}$.  Let $\bv{m}$ denote the $3 \times n$ matrix $(\bv{m})_{ki} = m_{ki}$.

Define
\[
\widetilde{A}_W(\bv{m})_{g,n} := \widetilde A_3(m_{11},\dots,m_{1n})_{g,n}\times_{\cM_{g,n,4}}\dots \times _{\cM_{g,n,4}} \widetilde A_1(m_{31},\dots,m_{3n})_{g,n}.
\]
$\widetilde{A}_W(\bv{m})_{g,n}$ carries three universal line bundles $\widetilde{\LL}_1,\widetilde{\LL}_2,\widetilde{\LL}_3$ satisfying 
\[
(\widetilde{\LL}_k)^{\otimes 4} \cong \omega_{\log}\left(- \sum_{i=1}^n 4m_{ki}D_i\right).
\] 
 
The above moduli space yields a uniform way of defining the extended FJRW correlators for $(\tilde E_7,G_{max})$.  Given $\phi_{h_1}, \ldots, \phi_{h_n} \in \sH_{W,G}^{ext}$, we define the $3\times n$ matrix 
\begin{equation*}
I(\bv{h}) =\left(\begin{matrix}
i_1(h_1)+\frac14 & \cdots & i_1(h_n)+\frac14\\
i_2(h_1)+\frac14 & \cdots & i_2(h_n)+\frac14\\
i_3(h_1)+\frac12 & \cdots & i_3(h_n)+\frac12
\end{matrix}\right).
\end{equation*}
Consider the following proposition.

\begin{proposition}\label{p:chern}  
On $\widetilde{A}_W(I(\bv{h}))_{0,n}$,  $\pi_*\big(\bigoplus_{k=1}^3\widetilde{\LL}_k\big)$ vanishes and $R^1\pi_*\big(\bigoplus_{k=1}^3\widetilde{\LL}_k\big)$ is locally free.  Furthermore, 
\[
\br{\tau_{a_1}(\phi_{h_1}),\dots, \tau_{a_n}(\phi_{h_n})}_{0,n}^{(\tilde E_7,G_{max})} =32 \int_{\tilde{A}_W(I(\bv{h}))_{0,n}} \prod \psi_i^{a_i}\cup c_{top}\Big(R^1\pi_*\big(\bigoplus_{k=1}^3\widetilde{\LL}_k\big)^\vee\Big).
\]
\end{proposition}

\begin{proof}
This proof is given in \cite{ChR}, and \cite{PrSh}, so we only give an outline. Comparing $A_3$ and $\widetilde{A}_3$, we see that if $m_{ki}\in \set{\frac14, \frac 12, \frac34 }$ for all $k, i$, then $h_1, \ldots, h_n \in G^{nar}$. In this case, we can identify $\widetilde{A}_4^4(\bv{m})_{g,n}$ with $A_4^4(\bv{m})_{g,n}$ via Lemma~\ref{l:linebndles} and $R^k\pi_*(\widetilde{\LL}_k) = R^k\pi_*(\LL_k)$.

We must consider the case where $h_i \in G \setminus G^{nar}$ for some $i$.  In this case $(I(\bv{h}))_{ki} = 1$ for some $k$.  Thus it suffices to prove that if $m_{ki} = 1$ for some $i$ and $k$, then
$\pi_*\big(\bigoplus_{k=1}^3\widetilde{\LL}_k\big) = 0$
and
$c_{top}\Big(R^1\pi_*\big(\bigoplus_{k=1}^3\widetilde{\LL}_k\big)\Big)=0.$

To do this assume $m_{k1}=1$, and consider the integer divisor $D_1$ on $\widetilde{A}_W(\bv{m})_{0,n}$ corresponding to the first marked point. 
We get the long exact sequence
\begin{align*}
0 &\to \pi_*(\widetilde\LL_k) \to \pi_*(\widetilde\LL_k(D_1)) \to \pi_*(\widetilde\LL_k(D_1)|_{D_1})\\
 &\to R^1\pi_*(\widetilde\LL_k) \to R^1\pi_*(\widetilde\LL_k(D_1))\to R^1\pi_*(\widetilde\LL_k(D_1)|_{D_1})\to 0.
\end{align*}
As in Lemma~\ref{p:concave}, the first two terms are 0. 

With $\pi_*(\widetilde{\LL}_k(D_1))$, there is one alteration. If $\cC$ is reducible, and $v '$ corresponds to the irreducible component carrying the first marked point, then $\deg \widetilde \cL_k(D_1)|_{\cC_{v '}}<\#\text{nodes}(\cC_{ v '})$.  But any section of $\widetilde \cL_k(D_1)$ must still vanish on all other components of $C$, and by degree considerations it must therefore vanish on $\cC_{v '}$.

$R^1\pi_*(\widetilde\LL_k(D_1)|_{D_1})$ also vanishes, so we have
\[
0 \to \pi_*\widetilde\LL_k(D_1)|_{D_1}\to R^1\pi_*\widetilde\LL_k\to R^1\pi_*\widetilde\LL_k(D_1)\to 0. 
\]
and 
\[
c_{top}(R^1\pi_*\widetilde\LL_k)=c_{top}(\pi_*\widetilde\LL_k(D_1)|_{D_1})\cdot c_{top}(R^1\pi_*\widetilde\LL_k(D_1)). 
\]
But $c_{top}(\pi_*\widetilde\LL_k(D_1)|_{D_1})=0$, as $\widetilde{\LL}_k(D_1)|_{D_1} \cong \LL_k|_{D_1}$ is a root of $\omega_{\log}|_{D_1}$ which is trivial. Thus $c_{top}(R^1\pi_*\widetilde\LL_k)=0$ as well. 
\end{proof}

We may define a $\CC^*$-equivariant generalization of the above theory.  This will allow us to compute correlators which, in the non-equivariant limit coincide with the genus zero FJRW correlators above.
Given a point $(\cC,p_1,\dots,p_n,\widetilde \cL_1,\widetilde \cL_1, \widetilde\cL_3)$ in $(\widetilde{A}_W)_{g,n}$, let $\CC^*$ act on the total space of $\bigoplus_{k=1}^3 \widetilde{\cL}_k$ by multiplication of the fiber.  This induces an action on $(\widetilde{A}_W)_{g,n}$.  

Set $R=H^*_{\CC^*}(pt,\CC)[[s_0,s_1, \dots]]$, the ring of power series in the variables $s_0,s_1,\dots$ with coefficients in the equivariant cohomology of a point,  $H^*_{\CC^*}(pt, \CC)=\CC[\lambda]$. 
Define a multiplicative characteristic class $\bc$ taking values in $R$, by
\[
\bc(E):=\exp\left(\sum_\ell s_\ell \ch_\ell(E) \right)
\]
for $E\in K^*((\widetilde A_W)_{g,n})$. 

Define the twisted state space
\[
\sH^{tw}:=\sH_{W,G}^{ext}\otimes R \cong \bigoplus_{h\in G} R\cdot \phi_h
\] 
and extend the pairing by 
\[
\br{\phi_{h_1},\phi_{h_2}}:=\begin{cases}
    \exp(-N_{h_1}s_0) & \text{ if } h_1=(h_2)^{-1}\\
    0 & \text{otherwise}.
    \end{cases}
\]

We may also define \emph{twisted correlators} as follows.  Given $\phi_{h_1}, \ldots, \phi_{h_n}$ basis elements in $\sH^{tw}$, define the invariant
\[
\br{\tau_{a_1}(\phi_{h_1}),\dots,\tau_{a_n}(\phi_{h_n})}_{g,n}^{tw} :=32 \int_{\tilde{A}_W(I(\bv{h}))_{g,n}} \prod \psi_i^{a_i}\cup \bv{c}\Big(R\pi_*\big(\bigoplus_{k=1}^3\widetilde{\LL_k}\big)\Big).
\]
taking values in $R$.  We can organize these correlators into generating functions $\cF_g^{tw}$ and $\cZ^{tw}$ 
as in Section~\ref{ss:cohft axioms}. It is not clear at first glance that these correlators come from a CohFT. We will see in the next section, however, that they do indeed.

\subsection{From twisted theory to FJRW theory}\label{section: the last one}
Specializing to particular values of $s_\ell$ yield different twisted correlators. One particularly important specialization is the following.  
From the partition function $\cZ^{tw}$, if we set 
\begin{equation}\label{e:sd}
s_\ell=\begin{cases}
-\ln \lambda & \text{if } \ell=0\\
\frac{(\ell-1)!}{\lambda^\ell} & \text{otherwise}
\end{cases}
\end{equation} 
we obtain the (extended) FJRW theory correlators defined above. To see this first consider the following lemma.

\begin{lemma}\cite[Lemma 4.1.2]{ChR}\label{l:ctop}
With $s_\ell$ defined as in \eqref{e:sd}, the multiplicative class $\bc(-V)=e_{\CC^*}(V^\vee)$. In particular, the non-equivariant limit $\lambda\to 0$
yields the top chern class of $V^\vee$.
\end{lemma}

By Proposition~\ref{p:concave}, $\pi_*(\bigoplus \widetilde{\LL}_k)=0$ and $\bc(R\pi_*(\widetilde{\LL}_k))=\bc(-R^1\pi_*(\widetilde{\LL}_k))$. Setting $s_\ell$ as in \eqref{e:sd} therefore yields 
\[
\bc\Big(R\pi_*\Big(\bigoplus_{k=1}^3\widetilde{\LL}_3\Big)\Big) = e_{\CC^*}\Big(R^1\pi_*\Big(\bigoplus_{k=1}^3\widetilde{\LL}_k\Big)^\vee\Big).
\] 
We have seen in Proposition~\ref{p:chern} that the FJRW correlators are obtained by the top chern class of $R^1\pi_*(\oplus_{k=1}^3\widetilde{\LL}_k)$, so we arrive at the following important result 
\begin{corollary}\label{c:nonequivlimit}
After specializing $s_\ell$ to the values in \eqref{e:sd}, 
\[
\lim_{\lambda\to 0} \cF_0^{tw}=\cF_0^{(\tilde E_7, G_{max})}.
\]
\end{corollary}

\section{Computing the $R$-matrix}\label{section: Rmatrix}

In this section, we will begin by describing the so-called ``untwisted'' theory, which we will see is equivalent to the product of GW theories of a point, as in Section~\ref{ss:gw theory of point}. From there we will show how to go from this basic theory to the FJRW theory of the pair $(\tilde E_7, G_{max})$, and then using Theorem~\ref{theorem: cylg via Givental}, we will obtain the $R$-matrix for the GW theory of $\PP_{4,4,2}$. 

\subsection{Untwisted theory}
In addition to the specialization mentioned at the end of the previous section, if we specialize to $s_\ell=0$ for all $\ell$, we obtain the ``untwisted'' theory. Using the projection formula, we can push all calculations down to $\overline{\cM}_{0,n}$ and obtain (cf. \cite{ChR})
\[
\langle \tau_{a_1}(\phi_{h_1}) \dots \tau_{a_n}(\phi_{h_n}) \rangle_{0,n}^{un} := 32\int_{\widetilde{A}_W(I(\bv{h}))_{0,n}} \psi_1^{a_1} \dots \psi_n^{a_n} =
\left(\begin{matrix}
\sum_i a_i\\
a_1,\dots,a_n
\end{matrix}\right)
\] 
whenever the line bundle degree axiom (axiom FJR~\ref{a:line bundle}) is satisfied.
From the untwisted theory, we obtain a CohFT. We will denote the generating functions of the untwisted theory by $\cF_g^{un}$ and $\cZ^{un}$.

Recall the function $\mathcal{T}^{(m+1)}$ of \eqref{eq: tauOfKdv}. Take $m+1= |G|$.
It's then connected to $\cF_0^{un}$ as follows.

\begin{proposition}\label{prop: Fun to Fkdv}
  The genus zero potential $\cF_0^{un}$ is obtained from $\res_\hbar \ln \mathcal{T}^{(|G|)}$ by a linear change of the variables.
\end{proposition}
\begin{proof}
  It can be checked explicitly that the function $\cF_0^{un}$ defines a semisimple algebra. It is also quasihomogeneous and so we can apply Theorem~\ref{theorem: teleman}. We only need to show that the upper-triangular group element of that theorem is trivial.
  
  Both functions $\mathcal{T}^{(|G|)}$ and $\cF_0^{un}$ are generating functions of the products of the CohFTs. 
  It is enough by considering the ``factors'' on the both sides. In particular we show that the untwisted theories of $A_3$ and $A_1$ are connected by the linear changes of the variables to the genus zero potentials of $\mathcal{T}^{(4)}$ and $\mathcal{T}^{(2)}$ respectively. 
  
Because of the topological recursion relation in genus zero it's enough to show this on the small phase space only. Namely, on the level of primary potentials. Let $F_0^{(KdV\otimes k)}$ and $F_0^{un\otimes r}$ stand for the primary genus zero potential of $\mathcal{T}^{(k)}$ and $A_r$ respectively. We have:
\[
F_0^{KdV \otimes 4} = \frac{u_0^3}{6}+\frac{u_1^3}{6}+\frac{u_2^3}{6}+\frac{u_3^3}{6}.
\]

Using the selection rule we have:
\[
F_0^{un\otimes 3} = \frac{1}{2} t_0 t_1^2 + \frac{1}{2} t_0^2t_2 + \frac{1}{2} t_0t_3^2 + \frac{1}{6} t_2^3 + t_1t_2t_3.
\]

One can check that the desired linear change of the variables is given by $u_0 = t_0-t_1+t_2-t_3$, $u_1 = t_0+t_1+t_2+t_3$, $u_2 = -t_0+t_2- \im(t_1-t_3)$, $u_3 = -t_0+t_2+\im(t_1-t_3)$.

\end{proof}

Recall from Corollary~\ref{c:nonequivlimit} that we obtain FJRW theory in genus zero as the non-equivariant limit of the twisted theory. In order to obtain the twisted theory, we use the following proposition. Recall that the Bernoulli polynomials are defined by the equation
\[
\frac{te^{xt}}{e^t-1}=\sum_{n=0}^{\infty}B_n(x)\frac{t^n}{n!}
\]

\begin{proposition}\label{p:sympl}
Recall the numbers $i_k(h)$ and $q_k$ from Section~\ref{section: twistedTheory}. Consider the upper-triangular group element $R^{tw}$ acting diagonally:
\[
 R^{tw} \left( \phi_h \right):= \prod_{k=1}^3 \exp\left(\sum_{\ell\geq 0} s_\ell\frac{B_{\ell+1}\big(i_k(h)+q_k\big)}{(\ell+1)!}z^\ell\right) \phi_h.
\] 
The action of this upper-triangular group element satisfies
\begin{equation}\label{eq:action of the cones}
\hat R^{tw}\cdot\cZ^{un}=\cZ^{tw}.
\end{equation}
\end{proposition}

\begin{proof}Note first that the identity 
$B_\ell(1-x)=(-1)^{\ell}B_\ell(x)$ implies $R^{tw}$ is in the upper-triangular Givental's group. 

Let the partition functions be written in the variables $t^{\ell,h}=q^{\ell,h}+\delta_{}$ with $h\in G$ and $\ell \ge 0$.
The proof is the same as the proof in \cite{ChZ,ChR,PrSh}. The basic idea is to consider both sides of \eqref{eq:action of the cones} as functions in the variables $s_\ell$. Then one shows that both sides satisfy 

\begin{equation}\label{e:diffeq}
\frac{\partial \Phi}{\partial s_\ell}=\sum_{k=1}^3 P_\ell^{(k)}\Phi
\end{equation}
where
\begin{align*}
  P_\ell^{(k)}=\frac{B_{\ell+1}(q_k)}{(\ell+1)!} &\frac{\partial}{\partial t^{\ell+1,\jw}}-\sum_{\stackrel{a \geq 0}{h \in G}} \frac{B_{\ell+1}(i_k(h)+q_k)}{(\ell+1)!}t^{a,h}\frac{\partial}{\partial t^{a+\ell,h}}
  \\ 
  +&\frac{\hbar}{2}\sum_{\stackrel{a+a'=\ell-1}{h, h' \in G}} (-1)^{a'}\eta^{h,h'}\frac{B_{\ell+1}(i_k(h)+q_k)}{(\ell+1)!} \frac{\partial^2}{\partial t^{a,h} \partial t^{a',h'}},
\end{align*}
and $\eta^{h,h'}$ denote the entries of the matrix inverse to the pairing. Since both $\cZ^{tw}$ and $\hat{R}^{tw}\cdot\cZ^{un}$ satisfy \eqref{e:diffeq} and both have the same initial condition (when $s_\ell=0$), they must then be equal.  

By the definition of quantization, it is clear that $\hat R^{tw}\cdot\cZ^{un}$ satisfies \eqref{e:diffeq}. 

That $\cZ^{tw}$ satisfies \eqref{e:diffeq} was proven by giving an expression for $\ch_\ell(R\pi_*(\widetilde{\LL}_k))$ using Grothendieck--Riemann--Roch. This was done in \cite{ChZ}, and generalized to the extended state space in \cite{ChR}.
\end{proof}

Consider $r_{GW} \in \mathrm{Hom}(\sH, \sH)[z]$ for $\sH = \sH^{ext}_{\tilde E_7,G_{max}}$ given by:
\[
r_{GW} \left( \phi_h \right):= \sum_{k=1}^3 \sum_{\ell\geq 0} s_\ell\frac{B_{\ell+1}\big(i_k(h)+q_k\big)}{(\ell+1)!}z^\ell \phi_h + \frac{1}{2}\left(\Gamma \left( \frac{3}{4} \right)\right)^4 \delta_{h\jw,\id} \ \phi_\jw.
\]
The following theorem gives a formula for the genus zero potential of the GW theory of $\PP^1_{4,4,2}$.
\begin{theorem}
For the upper-triangular group element $R_{GW} := \exp(r_{GW})$ and some lower-triangular group element $S$ the following holds:
\[
\cF_0^{\PP^1_{4,4,2}} = \lim_{\lambda\to 0}\res_{\hbar} \ln \left( \hat S^{-1} \cdot  \ \hat R_{GW} \cdot \cZ^{un} \right).
\]
\end{theorem}
Due to Proposition~\ref{prop: Fun to Fkdv} partition function $\cZ^{un}$ differs from the partition function $\mathcal{T}^{(m+1)}$ by a linear change of the variables and we get indeed the $R$-matrix reconstructing the GW theory of $\PP^1_{4,4,2}$ from the product of the GW theories of a point.

\begin{proof}
Let $\hat R := \hat R^\sigma$ and $\hat S := \hat S_0^c \cdot \hat S^\tau$ as in Theorem~\ref{theorem: cylg via Givental}.
By composing Corollary~\ref{c:nonequivlimit} and Proposition~\ref{p:sympl} we get:
\[
\cF_0^{(\tilde E_7,G_{max})} = \lim_{\lambda\to 0} \cF_0^{tw} = \lim_{\lambda\to 0} \res_{\hbar} \ln \left( \hat R^{tw} \cdot \cZ^{un} \right).
\]

From Theorem~\ref{theorem: cylg via Givental} we get:
\[
\cF_0^{(\tilde E_7,G_{max})} = \res_\hbar \ln \left( \hat R \cdot \hat S \cdot \cZ^{\PP^1_{4,4,2}} \right),
\]

Note that when considering Givental's action in genus zero, only the genus zero correlators of a CohFT given contribute to the Givental-transformed CohFT. So we can turn this around to obtain. 
\[
\cF_0^{\PP^1_{4,4,2}} = \res_\hbar \ln \left(\hat S^{-1} \cdot \hat R^{-1} \cdot \cZ^{(\tilde E_7,G_{max})} \right),
\]

Consider the extension of $R^\sigma$ to the state space $\cH^{tw}$. This can be done because $R^\sigma$ acts non-trivially only on the vector $\phi_{\jw^{-1}}$ belonging both to $\cH_{W,G}$ and $\cH^{tw}$. 

Slightly abusing the notation we denote by the same letter $r^\sigma$ the operator on $\cH^{tw}$, such that $r^\sigma (\phi_h) = \sigma \delta_{h\jw,\id} \ \phi_\jw$. And again $R^\sigma = \exp(r^\sigma z)$. It's clear that the action of the differential operator $\hat R^\sigma$ is not affected by the limit $\lambda \to 0$. The same is true of $\hat S^{-1}$.
\[
\cF_0^{\PP^1_{4,4,2}} = \lim_{\lambda\to 0}\res_{\hbar} \ln \left( \hat S^{-1} \cdot \hat R^{-1} \cdot  \ \hat R^{tw} \cdot \cZ^{un} \right).
\]

By comparing the formal power series in $\hbar$ we get:
\begin{align*}
\cF_0^{\PP^1_{4,4,2}} &= \lim_{\lambda\to 0}\res_{\hbar} \ln \left( \hat S^{-1} \cdot  \ \hat R^{-1} \cdot \hat R^{tw} \cdot \cZ^{un} \right),
    \\
    &= \lim_{\lambda\to 0} \res_{\hbar} \ln \left( \hat S^{-1} \cdot  \ \left( R^{-1} R^{tw} \right)\hat{\ } \cdot \cZ^{un} \right).
\end{align*}
This completes the proof.
\end{proof}

\appendix

\section{Gromov--Witten potential of $\PP^1_{4,4,2}$}\label{section: appendix}
In order to shorten the formulae let $t_k := t_{1,k}$ for $1 \le k \le 3$, $t_l := t_{2,l-3}$ for $4 \le l \le 6$, $t_7 := t_{3,1}$.
Let also $t_0$ correspond to the unit, $t_8$ to the hyperplane class of the cohomology ring of $\PP^1$ and $x = x(q)$, $y = y(q)$, $z = z(q)$, $w = w(q)$ be as in Section~\ref{section: GW of X4}.
Then the following expression for the genus zero GW ponential of $\PP^1_{4,4,2}$ was first announced by the first named author in \cite{B1}.

\begingroup
\everymath{\scriptstyle}
\tiny
\begin{align*}
  & F_0^{\PP^1_{4,4,2}} = -\frac{\left(x^6-5 x^4 y^2-5 x^2 y^4+y^6\right) }{4128768}\left(t_3^8+t_6^8\right)+\frac{x y \left(x^4+14 x^2 y^2+y^4\right) }{294912}t_3^2 t_6^2 \left(t_3^4+t_6^4\right)+\frac{z \left(8 x^4+8 y^4+19 z^4\right)}{294912} t_6^3 t_7 t_3^3
  \\
  &+\frac{x \left(x^2+y^2\right)^2 }{73728}\left(t_2 t_3^6+t_5 t_6^6\right)+\frac{y \left(x^2+y^2\right)^2 }{73728}\left(t_3^6 t_5+t_2 t_6^6\right)+\frac{5 x^2 y^2 \left(x^2+y^2\right) }{73728}t_6^4 t_3^4-\frac{\left(x^4-6 x^2 y^2+y^4\right) }{30720}\left(t_1 t_3^5+t_4 t_6^5\right)
  \\
  &-\frac{\left(x^4-3 x^2 y^2\right) }{3072}\left(t_2^2 t_3^4+t_5^2 t_6^4\right)+\frac{\left(3 x^2 y^2-y^4\right) }{3072}\left(t_3^4 t_5^2+t_2^2 t_6^4\right)+\frac{x y z\left(x^2+y^2\right) }{6144}t_3 t_6 \left(t_3^4+t_6^4\right) t_7+\frac{x^2 y \left(x^2+4 y^2\right)}{6144}t_3^2 t_6^2 \left(t_2 t_3^2+t_5 t_6^2\right)
  \\
  &+\frac{x y^2 \left(4 x^2+y^2\right) }{6144}t_3^2 t_6^2 \left(t_3^2 t_5+t_2 t_6^2\right)+\frac{x y \left(x^2+y^2\right) }{1536}\left(t_3^2 t_6^2 \left(t_1 t_3+t_4 t_6\right)+t_2 t_5 \left(t_3^4+t_6^4\right)\right)+\frac{x^2 y^2 }{1536}t_3 t_6 \left(t_3^3 t_4+t_1 t_6^3\right)
  \\
  &+\frac{x z\left(x^2+7 y^2\right) }{1536}t_3 t_6 t_7\left(t_3^2 t_5+t_2 t_6^2\right)+\frac{y z\left(7 x^2+y^2\right)}{1536}t_3 t_6 t_7\left(t_2 t_3^2+t_5 t_6^2\right)+\frac{x y \left(x^2+y^2\right)}{512} t_3^2 t_6^2\left(t_2^2+t_5^2\right) +\frac{x^2 y^2 }{384} \left(t_3^4+t_6^4\right) t_7^2
  \\
  &+\frac{x \left(x^2+y^2\right)}{384} \left(t_1 t_2 t_3^3+t_4 t_5 t_6^3\right)+\frac{y \left(x^2+y^2\right)}{384} \left(t_1 t_3^3 t_5+t_2 t_4 t_6^3\right)+\frac{\left(x^2+y^2\right) z}{384} t_7\left(t_3^3 t_4+t_1 t_6^3\right) +\frac{x^3}{384}\left(t_2^3 t_3^2+t_5^3 t_6^2\right)
  \\
  &+\frac{y^3 }{384} \left(t_3^2 t_5^3+t_2^3 t_6^2\right)-\frac{\left(3 w-x^2+2 y^2\right)}{384} \left(t_2^4+t_5^4\right)+\frac{x y^2 }{128} t_2 t_5 \left(t_3^2 t_5+t_2 t_6^2\right)+\frac{x^2 y }{128} t_2 t_5 \left(t_2 t_3^2+t_5 t_6^2\right)+\frac{x^2 y^2}{128} t_2 t_5 t_6^2 t_3^2
  \\
  &+\frac{x y \left(x^2+y^2\right)}{128} t_6^2 t_7^2 t_3^2+\frac{\left(2 x^2-y^2-3 w\right)}{96} t_7^4+\frac{x y^2}{64} t_3 t_6 \left(t_2 t_3 t_4+t_1 t_5 t_6\right)+\frac{x^2 y}{64} t_3 t_6 \left(t_3 t_4 t_5+t_1 t_2 t_6\right)
  \\
  &+\frac{x y z}{192} t_3 t_6 t_7 \left(3 t_2^2+3 t_1 t_3+3 t_5^2+3 t_4 t_6+4 t_7^2\right)+\frac{z\left(x^2+y^2\right)}{64} t_2 t_5 t_6 t_7 t_3-\frac{\left(w-x^2\right) }{64} \left(2 t_5^2 t_7^2+t_2^2 t_5^2+2 t_2^2 t_7^2\right)
  \\
  &-\frac{\left(2 w-x^2+y^2\right)}{64} \left(t_1^2 t_3^2+t_4^2 t_6^2\right)+\frac{x y^2}{32} \left(t_2 t_7^2 t_3^2+t_5 t_6^2 t_7^2\right)+\frac{x^2 y }{32} \left(t_5 t_7^2 t_3^2+t_2 t_6^2 t_7^2\right)+\frac{x y }{32} \left(2t_1 t_2 t_5 t_3+t_1^2 t_6^2+t_4^2 t_3^2\right)
  \\
  &-\frac{w}{32} \left(t_4 t_5^2 t_6+t_1 t_2^2 t_3\right)+\frac{\left(x^2-y^2-w\right)}{32} \left(t_1 t_5^2 t_3+ t_2^2 t_4 t_6\right)-\frac{\left(w-x^2\right)}{16} \left(t_1 t_7^2 t_3+ t_1 t_4 t_6 t_3+ t_4 t_6 t_7^2\right)+\frac{x y }{16} t_2 t_5\left(t_4t_6+2 t_7^2\right)
  \\
  &+\frac{x z}{16} t_7\left(t_2 t_4 t_3+t_1 t_5 t_6\right)+\frac{y z }{16} t_7 \left(t_3t_4 t_5+t_1 t_2 t_6\right)+\frac{x}{8} \left(t_1^2 t_2+t_4^2 t_5\right)+\frac{y}{8} \left(t_2 t_4^2+ t_1^2 t_5\right)+\frac{z}{4} t_1 t_4 t_7
  \\
  &+\frac{1}{8} t_0 \left(t_2^2+t_5^2+2 t_7^2+2 t_1 t_3+2 t_4 t_6\right)+\frac{1}{2} t_0^2 t_8.
\end{align*}
\endgroup



\end{document}